\documentclass[twoside,11pt]{amsart}
\usepackage{graphicx,color, enumerate}
\usepackage{amssymb}
\usepackage{epstopdf, xypic}
\usepackage{booktabs}
\newtheorem{thm}{Theorem}[section]
\newtheorem*{mainthm}{Main Theorem}
\newtheorem{Notation}[thm]{Notation}
\newtheorem{prop}[thm]{Proposition}
\newtheorem{Ass}[thm]{Assumptions}
\newtheorem{lem}[thm]{Lemma}
\newtheorem{cor}[thm]{Corollary}

\newtheorem{qst}[thm]{Question}

\theoremstyle{definition}

\theoremstyle{remark}
\newtheorem{rmk}[thm]{Remark}

\def\QED{\hfill$\square$}

\renewcommand{\emptyset}{\varnothing}

\renewcommand{\phi}{\varphi}

\newcommand{\coker}{\mathop{\mathrm{Coker}}\nolimits}

\renewcommand{\i}[1]{\mathfrak{#1}}
\newcommand{\p}{\i{p}}

\newcommand{\m}{\i{m}}

\newcommand{\depth}{\mathop{\mathrm{depth}}\nolimits}

\newcommand{\h}{\mathop{\mathrm{ht}}\nolimits}

\newcommand{\pd}{\mathop{\mathrm{pd}}\nolimits}

\renewcommand{\*}{\bullet}

\def\Lra{\Longrightarrow}
\def\lra{\longrightarrow}

\title{The projective dimension of three cubics is at most 5} 

\author{Paolo Mantero}
\address{University of Arkansas, Department of Mathematical Sciences, Fayetteville, AR 72701}
\email{pmantero@uark.edu}
\author{Jason McCullough}
\address{Iowa State University, Department of Mathematics, Ames, IA 50011}
\email{jmccullo@iastate.edu}

\begin{document}

\begin{abstract}
Let $R$ be a polynomial ring over a field and $I$ an ideal generated by three forms of degree three. Motivated by Stillman's question, 
Engheta proved that the projective dimension $\pd(R/I)$ of $R/I$ is at most 36, although the example with largest projective dimension he constructed has $\pd(R/I)=5$.  Based on computational evidence, it had been conjectured that $\pd(R/I)\leq 5$. In the present paper we prove this conjectured sharp bound.
\end{abstract}

\maketitle

\section{Introduction}
 Let $R = k[x_1,\ldots,x_N]$ be  a polynomial ring over a field $k$, and let $I = (f_1,\ldots,f_n)$ be a homogeneous ideal. The projective dimension of $R/I$ is one of the possible measures of the complexity of an ideal. Hilbert's Syzygy Theorem implies that $\pd(R/I) \leq N$; this was further refined in a celebrated theorem by Auslander and Buchsbaum, who proved $\pd(R/I)=N-\depth(R/I)$.  Motivated by computational efficiency issues, Stillman \cite{PS} asked whether one can find an upper bound for $\pd(R/I)$ solely based on the number of minimal generators of $I$ and their degrees. More precisely he asked:
\begin{qst}[{Stillman's question \cite[Problem 3.14]{PS}}]\label{Stillman} Is there a bound on $\pd(R/I)$ depending only on $d_1,\ldots,d_n$, and $n$, where $d_i = \deg(f_i)$?
\end{qst}
\noindent Additional motivation was provided by Caviglia, who showed the equivalence of Question \ref{Stillman} with the analogous question (\cite[Question~3.15]{PS}) for Castelnuovo-Mumford regularity. 

A positive answer to Question \ref{Stillman} in its full generality has been recently proved by Ananyan and Hochster \cite{AH2}.  However, these bounds are very large and in most cases not explicit, so the quest for optimal bounds is still wide open. The following are the known results in this direction:
\begin{itemize}
 \item When $I$ is generated by quadrics and ${\rm ht}(I) = 2$, then $\pd(R/I) \le 2n - 2$ and this bound is sharp \cite{HMMS1}. 
 \item When $I$ is generated by four quadrics, then $\pd(R/I) \le 6$ and this bound is sharp \cite{HMMS3}.
 \item When $I$ is generated by $3$ cubics, Engheta in a series of three papers \cite{En1} -- \cite{En3} proved $\pd(R/I) \le 36$. However the example with the largest projective dimension he could construct has $\pd(R/I)=5$.  
\end{itemize}
 During the special program in Commutative Algebra at the MSRI in 2012 -- 2013, strong computational efforts were made to find examples of ideals $I$ generated by 3 cubics with large projective dimension, which led to new classes of ideals with $\pd(R/I)=5$. The simplest of these examples is the ideal $I = (x^3, y^3, x^2a+xyb+y^2c)$ in the polynomial ring $K[x,y,a,b,c]$.  One checks that $x^2y^2 \in I:(x,y,a,b,c) - I$ and hence $\pd(R/I) = 5$.  However, no ideals were found with $\pd(R/I)>5$.  The main goal of this paper is to complete the project started by Engheta and show that $5$ is the optimal upper bound.
 
 Our main result is the following:
 
\begin{mainthm}\label{main}
Let $R$ by a polynomial ring over a field $k$, and let $I$ be generated by $3$ homogeneous polynomials of degree $3$.  Then ${\rm pd}(R/I)\leq 5$. 
\end{mainthm}

Since the projective dimension does not change under faithfully flat extensions, one may always assume $k=\overline{k}$ is algebraically closed. 
Our general approach is similar to Engheta's; we divide the proof first by height, then by multiplicity, then by considering the primary decompositions of certain ideals associated to $I$.  
The proofs are then necessarily long and technical; we believe this is a necessary price to pay to get an optimal bound.  Our paper uses several new tools not available to Engheta including:
structure theorems for ideals of low multiplicity that are primary to a linear prime of height two, which we develop in a separate paper \cite{MM};  similar structure theorems for ideals $L$ linked to $I^{un}$ (see Section 2 for the definitions); and considerations on Hilbert function and projective dimensions (e.g. Propositions~\ref{linkage}(d) and ~\ref{<=t}). 

 The paper is organized as follows.  In Section 2, we collect many techniques and results employed several times throughout the paper, we recall that $e(R/I)\leq 7$ and we present a table summarizing all the cases needed for the proof of the Main Theorem. 
  In Section 3 we prove that if $e(R/I)\leq 3$, then $\pd(R/I)\leq 4$. In Section 4 we prove the Main Theorem in the case $e(R/I) = 6$.  Section 5 contains the proofs of the cases $e(R/I)= 4$ and $e(R/I) = 5$.  

\section{Notation and general results}

Throughout this paper we use the following notation:

\begin{Notation}\label{not}
$R = \oplus_{i \ge 0} R_i$ is a standard graded polynomial ring over a field $k$, which, without loss of generality, we may assume is algebraically closed. Moreover, 
\begin{itemize}
\item $\m = \oplus_{i \ge 1} R_i$ is the unique homogeneous maximal ideal of $R$;
\item $I=(f_1,f_2,f_3)$ is generated by three forms with $\deg(f_i)=3$;
\item It is well-known that it suffices to prove the Main Theorem when $I$ has height two $($e.g. see \cite[Remark~2]{En1}$)$, so  after possibly taking linear combinations of the generators we may assume 
\[ \h(f_i,f_j)=2 \qquad \mbox{ for every }i\neq j;\]
\item $L=(f_1,f_2):f_3$ is a link of $I^{un}$ $($see discussion before Theorem \ref{linkage} for the definition of a link$)$;
\item $e(R/J)$ denotes the multiplicity of $R/J$;
\item we call an ideal $J$ {\em unmixed (of height $h$)} if ${\rm ht}(\mathfrak{p}) = h$ for every $\p\in {\rm Ass}(R/J)$.  
\item We denote by  $I^{un}$ the {\em unmixed part of $I$}, which is the intersection of the primary components of $I$ corresponding to the minimal primes of $I$ of minimal height. In particular, $I^{un}$ is an unmixed ideal, with $I\subseteq I^{un}$, ${\rm ht}(I)={\rm ht}(I^{un})$ and $e(R/I)=e(R/I^{un})$ $($see Prop. \ref{Ass}$)$.
\item For any homogeneous ideal $J$ and $n\in \mathbb{N}_0$, we denote by $[J]_n$ $([J]_{\ge n}$, resp.$)$ the $R$-ideal generated by all forms of degree $n$ $($at least $n$, resp.$)$ in $J$.
\end{itemize}
\end{Notation}
Furthermore, by a {\it quadric} we mean a form of degree 2, and by a {\it cubic} we mean a homogeneous polynomial of degree 3.

%

\begin{prop}[Associativity Formula]\label{Ass}
If $J$ is an ideal of $R$, then $$e(R/J)=\sum_{\substack{\rm{primes\,\, } \mathfrak{p} \supseteq J\\ \rm{ht}(\mathfrak{p}) = \rm{ht}(J)}}e(R/\mathfrak{p})\lambda(R_\mathfrak{p}/J_\mathfrak{p}),$$
where $\lambda(R_\mathfrak{p}/J_\mathfrak{p})$ denotes the length of the Artinian ring $R_\mathfrak{p}/J_\mathfrak{p}$.
\end{prop}
Let $J=q_1\cap \ldots \cap q_r$ be the irredundant primary decomposition of an unmixed ideal $J$ of $R$, and let $\p_i=\sqrt{q_i}$. By the Associativity Formula, one has $e(R/J)=\sum_{i=1}^re_i\lambda_i$, where $e_i=e(R/\p_i)$ and $\lambda_i=\lambda((R/J)_{\p_i})$ for every $i$. 
We adopt the notation of Engheta \cite{En3} and say that $J$ is of type \[\langle\underline{e};\underline{\lambda}\rangle = \langle e_1,\ldots,e_r; \lambda_1,\ldots,\lambda_r\rangle.\] 
For instance, a prime ideal of multiplicity two is of type $\langle2;1\rangle$; an ideal of multiplicity three primary to a linear prime is of type $\langle 1;3 \rangle$; on the other hand, if $J$ is an unmixed ideal of type $\langle 1,2;\,3, 1\rangle$, then $J=q_1\cap q_2$, where $q_1$ is an ideal of multiplicity 3 primary to a linear prime and $q_2$ is a prime ideal of multiplicity two.

Table~\ref{table2} summarizes all the cases needed to prove the Main Theorem. On the leftmost column we list the seven possible values for the multiplicity $1\leq e(R/I)\leq 7$; the next column lists all the possible types of $I^{un}$, or when $e(R/I) = 5$ or $6$, the type of $L = (f_1,f_2):I$; the third column from the left contains the upper bound for $\pd(R/I)$ in that specific subcase; the fourth column contains the reference for the proof.
 
\begin{table}
\begin{tabular}{|l|l|l|l|}
\hline
{$\mathbf{e}(R/I)$} & 
$\langle\underline{e};\underline{\lambda}\rangle$ for $I^{un}$& Bound for $\pd(R/I)$ & Justification\phantom{lotsofextra words}\\
\hline\hline
{\bf 1}&$\langle1;1\rangle$&3& $I^{un}$ is CM (Thm \ref{pd+1}(i))\\
\hline
{\bf 2}&$\langle2;1\rangle$& 3 & $I^{un}$ is CM (Thm \ref{pd+1}(i))\\
&$\langle1;2\rangle$ & 4 & quadric in $I^{un}$ (Thm \ref{pd+1}(iii))\\
&$\langle1,1;1,1\rangle$& 4 & quadric in $I^{un}$ (Thm \ref{pd+1}(iii))\\
\hline
{\bf 3} & $\langle3;1\rangle$ & 3 & $I^{un}$ is CM (Thm \ref{pd+1}(i)) \\
&$\langle1;3\rangle$ & 4  &  Proved in \cite[Prop. 5.2]{MM} \\
&$\langle1,2;1,1\rangle$ & 4 & quadric in $I^{un}$ (Thm \ref{pd+1}(iii)) \\
&$\langle1,1;1,2\rangle$ & 4 & Prop \ref{L11;12} \\
&$\langle1,1,1;1,1,1\rangle$ & 4 & Prop \ref{111;111} \\
\hline
{\bf 4}&$\langle4;1\rangle$ & 5 & Proved in \cite[Sec. 2.1]{En3}\\ 
&$\langle2;2\rangle$ & 4& quadric in $I^{un}$ (Thm \ref{pd+1}(iii)) \\
&$\langle1;4\rangle$ & 5  & Proved in \cite[Prop. 5.4]{MM} \\
&$\langle1,3;1,1\rangle$ & 5  & Prop.~\ref{13;11} \\
&$\langle1,1;1,3\rangle$ & 5  & Prop.~\ref{11;13}\\
&$\langle2,2;1,1\rangle$ & 4  & quadric in $I^{un}$ (Thm \ref{pd+1}(iii))\\
&$\langle1,2;2,1\rangle$ & 5   & Prop. \ref{12;21} \\
&$\langle1,1;2,2\rangle$ & 5  &  Prop. \ref{11;22} \\
&$\langle1,1,2;1,1,1\rangle$ & 5 & Prop. \ref{112;111}\\
&$\langle1,1,1;1,1,2\rangle$ & 5  & Prop. \ref{111;112} \\
&$\langle1,1,1,1;1,1,1,1\rangle$ & 5  & Prop. \ref{1111;1111}\\
\hline
{\bf 7}& {\rm any}& 2  & $I$ is CM (Cor~\ref{e=7})\\
\hline
\end{tabular}\\
\vspace{.5cm}
\begin{tabular}{|l|l|l|l|}
\hline
{$\mathbf{e}(R/I)$} & 
$\langle\underline{e};\underline{\lambda}\rangle$ for $L$& Bound for $\pd(R/I)$ & Justification\phantom{lotsofextra words}\\
\hline\hline

{\bf 5}&$\langle4;1\rangle$ & 5 & Proved in \cite[Sec. 2.2]{En3} \\ 
&$\langle2;2\rangle$ & 4& quadric in $L$ (Thm \ref{pd+1}(iii)) \\
&$\langle1;4\rangle$ & 5  & Proved in \cite[Prop. 5.5]{MM} \\
&$\langle1,3;1,1\rangle$ & 5  & Prop.~\ref{L13;11} \\
&$\langle1,1;1,3\rangle$ & 5  & Prop.~\ref{L11;13} \\
&$\langle2,2;1,1\rangle$ & 4  & quadric in $L$ (Thm \ref{pd+1}(iii)) \\
&$\langle1,2;2,1\rangle$ & 5 & Prop. \ref{12;21}\\
&$\langle1,1;2,2\rangle$ & 5  &  Prop. \ref{11;22} \\
&$\langle1,1,2;1,1,1\rangle$ & 5 & Prop. \ref{112;111} \\
&$\langle1,1,1;1,1,2\rangle$ & 5 & Prop. \ref{111;112} \\
&$\langle1,1,1,1;1,1,1,1\rangle$ & 5  & Prop. \ref{1111;1111}\\
\hline
{\bf 6}&  $\langle3;1\rangle$ & 3 & $I^{un}$ is CM (Thm \ref{pd+1}(i)) \\
&$\langle1;3\rangle$ & 5 & Proved in \cite[Prop. 5.3]{MM} \\
&$\langle1,2;1,1\rangle$ & 4 & Proved in \cite[Sec. 2.3]{En3}\\
&$\langle1,1;1,2\rangle$ & 5  & Proved in \cite[Sec. 2.3]{En3} \\
&$\langle1,1,1;1,1,1\rangle$ & 4  & Prop \ref{111;111}\\

\hline

\end{tabular}
\caption{Bounds on $\pd(R/I)$ where $I$ is a height $2$ ideal generated by $3$ cubics}
\label{table2}
\end{table}

Next, we recall the notion of link of an unmixed ideal $J$. Given any homogeneous regular sequence $C\subseteq J$ of $\h(J)$ elements, the ideal $L=C:J$ is called a {\em a link of $J$} ({\em by} $C$), and the ideals $J$ and $L$ are said  to be {\it linked} by $C$, denoted $J\stackrel{C}{\sim} L$ (or simply $J\sim L$). It is known that $L$ is unmixed of $\h(L)=\h(J)$ and $C:L=J$ \cite{PSz}.
Moreover, the following linkage results are well-known and will be used frequently.
\begin{thm}\label{linkage}
If $J\stackrel{C}{\sim}H \stackrel{D}{\sim} L$, then
\begin{itemize}
\item[$($a$)$] $($\cite{PSz}$)$ $R/J$ is Cohen-Macaulay if and only if $R/H$ is Cohen-Macaulay;
\item[$($b$)$] $($\cite{PSz}$)$ $e(R/J)+e(R/H)=e(R/C)$;
\item[$($c$)$] $($\cite{Du}, \cite[Theorem~1.1]{AN}$)$ $\pd(R/J)=\pd(R/L)$;
\item[$($d$)$] $($cf. \cite[Proposition~5.1]{MM}$)$ Write $C=(c_1,\ldots,c_g)$ and $D=(d_1,\ldots,d_g)$. If, up to reordering, $\deg(c_i)=\deg(d_i)$ for every $i$, then $R/J$ and $R/L$ have the same Hilbert function. In particular, the minimal degree of a minimal generator and the number of generators of minimal degree are the same for $J$ and $L$.
\end{itemize}
\end{thm}

We also recall previous results of Engheta which will be employed. Recall that $J^{un}$ denotes the unmixed part of an ideal $J$.
\begin{thm}\label{pd+1}
Let $I$ and $L$ be as in Notation \ref{not}. Then 
\begin{enumerate}[(i)]
\item $($\cite[Theorem 7]{En1}$)$ $\pd(R/I)\leq \pd(R/L)+1$. In particular, if $I^{un}$ or $L$ is Cohen-Macaulay, then $\pd(R/I)\leq 3$.
\item If either $I^{un}$ or $L$ contains a linear form, then $\pd(R/I)\leq 3$.
\item $($\cite[Theorems~16~and~17]{En1}$)$ If either $I^{un}$ contains a quadric, or $e(R/I)\leq 5$ and $L$ contains a quadric, then $\pd(R/I)\leq 4$.
\end{enumerate}
\end{thm}

\begin{proof} (i) and (iii) are proved in the referenced results. (ii) was proved by B. Engheta when $I^{un}$ contains a linear form, \cite[Proposition 6]{En1}. So, assume $L$ contains a linear form. Then $L$ contains a complete intersection of degrees $1$ and $3$. Hence $e(R/L)\leq 3$. When $e(R/L)\leq 2$, either $L$ or an ideal linked to $L$ has multiplicity $1$, therefore it is Cohen-Macaulay (e.g. by Proposition \ref{non-deg}(2)).  If $e(R/L)=3$, then $L$ is a complete intersection (e.g. by \cite[Lemma~8]{En1}). In either case we obtain $\pd(R/I)\leq 3$ by part (i). 
\end{proof}

We recall here a standard application of the Depth Lemma:
\begin{lem}\label{ses}
Let $J$ be an ideal in a polynomial ring $R$.
\begin{itemize}
\item[$($1$)$] If $h$ is any element of $R$, then 
\[\pd (R/J) \leq \max \{ \pd (R/(J:h)), \pd (R/(J+(h)))\}.\]
\item[$($2$)$] If $J=J_1\cap J_2$, then 
\[\pd(R/J)\leq \max \{ \pd (R/J_1), \pd(R/J_2), \pd (R/(J_1+J_2))-1\}.\]
\end{itemize}
\end{lem}

Recall that a {\it linear prime} is a prime ideal generated by linear forms. Since $k=\overline{k}$, the following classical results can be employed
\begin{prop}(cf. \cite[Proposition~0]{EH2}, \cite[Theorem~1]{EH2})\label{non-deg}
Let $\mathfrak{p}$ be a homogeneous prime ideal 
\begin{enumerate}
\item If $e(R/\p)\leq \h(\p)$, then $\p$ contains a linear form.
\item In particular, if $e(R/\p)=1$, then $\p$ is a linear prime; if $e(R/\p)=\h(\p)=2$, then $\p=(\ell,q)$ for a linear form $\ell$ and an irreducible quadric $q$, so $\p$ is a complete intersection.
\item (Del-Pezzo Bertini) If $\h(\p)=2$ and $e(R/\p)=3$, then $R/\p$ is Cohen-Macaulay.
\end{enumerate}
\end{prop}

\subsection{Additional results} We now collect several results that will be used throughout the paper. 

\begin{rmk}\label{observ}
$($a$)$ The cubics in $I^{un}$ and $L$ 
generate an ideal of height two; moreover $I^{un}$ contains at least three linearly independent cubics.

\noindent $($b$)$ If $f\in R$ and $J_1,J_2$ are ideals, then 
\[J_1\cap \left(J_2+(f)\right)=J_1\cap \big[J_2 + f\big((J_1+J_2):f\big)\big].\]

\noindent $($c$)$ If every cubic in $I^{un}$ can be written in terms of at most $t$ linear forms, then $\pd(R/I)\leq t$.  In this case $I$ is extended from a $t$-variable polynomial ring.  The result then follows from Hilbert's Syzygy Theorem and the fact that polynomial extensions are flat.

\noindent $($d$)$ If $C$ is a complete intersection of height $2$ inside $I$, then $C:I=C:I^{un}$.

\noindent $($e$)$ If $J_1,J_2$ are ideals and $f\in R$, then $(fJ_1 + J_2):f=J_1+(J_2:f)$.

\noindent $($f$)$ If $J_1,J_2$ are ideals and $f\in J_1$, then $(J_1 \cap J_2)+(f)=J_1\cap (J_2+(f))$
\end{rmk}

In the next result we generalize Remark~\ref{observ}(c) to the case where every generator of either $I^{un}$ or an ideal linked to it can be written in terms of a regular sequence of at most $t$ forms:
\begin{prop}\label{<=t}
Let $J$ be an almost complete intersection ideal of height $g$ in $R=k[x_1,\ldots,x_n]$. If there exists a regular sequence of forms $h_1,\ldots,h_t$ such that either $J^{un}$ or some ideal linked to $J^{un}$ is extended from $A=k[h_1,\ldots,h_t]$, then $\pd(R/J)\leq {\rm max}\{g+1,t\}$.
\end{prop}
\begin{proof}
If $J^{un}$ is an $R$-ideal extended from $A$, we can write $J^{un}=J_0R$ for some $A$-ideal $J_0$. Let $C\subseteq J_0$ be a complete intersection of height $g$ and set $L_0=C:_AJ_0$ and $L'=L_0R$, then $J^{un}\sim L'$. So under either assumption $J^{un}$ is linked to an ideal $L'=L_0A$ extended from $A=k[h_1,\ldots,h_t]$, and since $\h(L_0)=g$, then $t\geq g$. If $t=g$, then $L_0$ is Cohen-Macaulay and so is $L'$, thus $\pd(R/J)\leq g+1$ by \cite[Theorem 7]{En1}. If $t\geq g+1$, since $L_0$ is unmixed of height $g$ in $A$,  then ${\rm depth}(A/L_0)\geq 1$, which by the Auslander-Buchsbaum formula yields $\pd(R/L')=\pd(A/L_0)\leq t-1$. By \cite[Theorem 7]{En1}, we have $\pd(R/J)\leq \pd(R/L')+1\leq t$.
\end{proof}


Recall that $[B]_n$ denotes the ideal generated by all forms of degree $n$ in a given homogeneous ideal $B$. The following lemma will be employed several times to obtain a more explicit description of certain intersections.
\begin{lem}\label{cubic}
Let $H$ be a homogeneous $R$-ideal.
\begin{enumerate}[(i)]
\item Let $K'$ be another homogeneous ideal, $f\in R_3$ and $K=(f)+K'$. \\
If $f\notin H+K'$, then $\left[K\cap H\right]_3 = \left[K'\cap H\right]_3$. \\
If $f\in H+K'$, then there exists $f'\in R_3$ such that $$[K \cap H]_3 =  \left[(f')+(K' \cap H)\right]_3.$$

\item Let $\p=(x,y)$ be a linear prime, let $H\not\subseteq \p$ be an ideal, and let $a,b$ be forms with 
${\rm ht}(x,y,a,b)=4$. If $ax+by\in \p^2+H$, then 
\[\big(\p^2+(ax+by)\big)\cap H = \big(a'x+b'y\big)+\big(\p^2\cap H\big), \]
where $a',b'$ are forms with $\deg(a')=\deg(b')=\deg(a)$  and $(x,y,a',b')=(x,y,a,b)$.
\end{enumerate}

\end{lem}

\begin{proof} 
%
(i) If $f\notin K'+H$, then $I'=(K'+H):f\subseteq \m$, and thus $fI'$ is generated in degree at least four. Applying Remark \ref{observ}(b) we obtain $[K\cap H]_3=[(K'+(f)I') \cap H]_3$ and since $[(f)I']_3=[0]_3$, it follows that $[K\cap H]_3=[K' \cap H]_3$.
If $f\in K' + H$, then there exists $F\in [K']_3$ such that $f'=f+F\in H$, thus $f'\in K\cap H$ and $K=(f')+K'$. By the modularity law, $K\cap H=(f') + (K'\cap H)$ and the statement follows.

%

(ii) Set $K=\p^2+(ax+by)$. By assumption we can write $ax+by=g + h$ for some $g\in \p^2$ and $h\in H$. Then $h=ax+by-g\in K\cap H$; writing $g=r_1x^2+r_2xy+r_3y^2$ for forms $r_i\in R$, we see $h=a'x+b'y$, where $a'=a-r_1x-r_2y$ and $b'=b-r_3y$. In particular, $(a',b',x,y)=(a,b,x,y)$. Moreover, since $a'x+b'y=ax+by$ modulo $\p^2$, we can write $K=\p^2+(a'x+b'y)$. By the modularity law we now obtain
$$J=\big[\p^2+(a'x+b'y)\big]\cap H=(a'x+b'y)+\big[\p^2\cap H\big].$$
\end{proof}

Next, we prove an upper bound of $\pd(R/J)\leq 4$ when $J$ is generated by 2 quadrics and one cubic. This is sharp since $\pd(R/(x^2,y^3, ax+by))=4$, where $R = K[x,y,a,b]$.
\begin{prop}\label{223}
Let $J=(q_1,q_2,c)$, where $q_1,q_2$ are quadrics and $c$ is a cubic. Then $\pd(R/J)\leq 4$.
\end{prop}

\begin{proof}
If $\h(J)=3$, then $\pd(R/J)=3$ because $J$ is a complete intersection. If $\h(J)=1$, then $J=\ell J_1$, where $J_1=(x,y,q)$ is a complete intersection of height 3. Then $\pd(R/J)=\pd(R/J_1)=3$. 

We may then assume $\h(J)=2$. If $\h(q_1,q_2)=1$, then we can write $J=(xy, xz, c)$ for linear forms $x,y,z$ with $c\notin (x)$. Then by Remark \ref{observ}(e) we have $\pd(R/(J:x))=\pd(R/(y,z,c))\leq 3$, and $\pd(R/(J+(x)))=\pd(R/(x,c))=2$; so Lemma \ref{ses}(1) yields $\pd(R/J)\leq 3$. 

If $\h(q_1,q_2)=2$, then $e(R/J)\leq 4$. If $e(R/J)=1,4$, then $J^{un}$ is a complete intersection. If $e(R/J)=3$, then by Theorem \ref{linkage}(b) the link $L'=(q_1,q_2):J^{un}$ of $J^{un}$ is an unmixed ideal with $e(R/L')=1$, so by Proposition~\ref{non-deg}(2), $L'$ is a complete intersection. In all these cases \cite[Theorem~7]{En1} implies $\pd(R/J)\leq 3$.
If $e(R/I)=2$, then by Theorem \ref{linkage}(b) one has $e(R/L')=2$, thus $\pd(R/L')\leq 3$ by \cite[Proposition~11]{En1}, and now \cite[Theorem~7]{En1} yields $\pd(R/J)\leq 4$.
\end{proof}

Next, we prove the desired bound when two of the three generators lie in a principal ideal.
\begin{cor}\label{x}
If we can write $I=(f_1',f_2',f_3')$ where $(f_1',f_2')\subseteq (g)$ for some non-unit $g\in R$, then $\pd(R/I)\leq 4$.
\end{cor}

\begin{proof}
Observe that $\pd(R/(I+(g)))=\pd(R/(g,f_3'))=2$, so by Lemma \ref{ses}(1) it suffices to prove $\pd(R/I:g)\leq 4$. 
By Remark \ref{observ}(a), $\h(g,f_3')=2$, and $1\leq \deg(g)\leq 2$. If $g\in R_1$, then $(g)$ is a prime ideal.  Write $f_i'=gq_i$, where $q_i$ is a quadric for $i=1,2$. Then $I:g=(q_1,q_2)+\left((f_3'):g\right)=(q_1,q_2,f_3')$. By Proposition \ref{223}, $\pd(R/I:g)\leq 4$.
If $g\in R_2$, write $f_1'=xg$ and $f_2'=yg$ for linearly independent forms $x,y\in R_1$. Then $I:g=(x,y)+\left((f_3'):g\right)=(x,y,f_3')$, so one has $\pd(R/I:g)\leq 3$. 
\end{proof}

The first part of the following corollary was first proved by Engheta \cite{En2} and later generalized in \cite{HMMS2}. The second part of the statement was proved in \cite{HMMS2}.
\begin{cor}[{cf. \cite[Corollaries 2.3 and 2.9]{HMMS2}}]\label{e=7}
Let $I$ be an almost complete intersection ideal generated by 3 cubic forms. Then $e(R/I)\leq 7$.  Moreover, if $e(R/I)=7$, then $I$ is Cohen-Macaulay and hence $\pd(R/I)=2$.
\end{cor}

It is thus natural to divide the proof of the Main Theorem by the multiplicity of $R/I$. By the above, we only need to consider the cases where $1\leq e(R/I)\leq 6.$

\section{The cases $1\leq e(R/I)\leq 3$}

The cases of multiplicity one and two were proved by Engheta \cite{En1}.  We include a proof for completeness. 
\begin{prop}\label{e=1,2}
If $e(R/I)=1$, then $\pd(R/I)\leq 3$. If $e(R/I)=2$, then $\pd(R/I)\leq 4$.
\end{prop}

\begin{proof}

If $e(R/I)=1$, then $e(R/I^{un})=1$, so by Proposition \ref{non-deg}(2) $I^{un}$ is Cohen-Macaulay; thus Theorem \ref{pd+1}(i) yields $\pd(R/I)\leq 3$. When $e(R/I)=2$, by the Associativity Formula \ref{Ass}, either $I^{un}$ is prime, or is primary to a linear prime, or is the intersection of two linear primes. If $I^{un}$ is prime, by Proposition~\ref{non-deg}(2) it is a complete intersection; then, by Theorem~\ref{pd+1}(i) one has $\pd(R/I)\leq 3$. If instead $I^{un}$ is primary to a linear prime or $I^{un}$ is the intersection of two linear primes, then $I^{un}$ contains a quadric by \cite[Prop. 8]{En1}, hence $\pd(R/I)\leq 4$ by Theorem~\ref{pd+1}(iii).
\end{proof}

The case of multiplicity three is more involved.
\begin{rmk}\label{shapee=3}
By the Associativity Formula \ref{Ass}, an unmixed ideal $H$ of height two and multiplicity three has one of the following five forms:
\begin{enumerate}
\item $H$ is a prime ideal (corresponding to the type $\langle 3;1\rangle$);
\item $H$ is primary to a linear prime (i.e. type $\langle 1;3\rangle$);
\item $H$ is the intersection of two prime ideals (i.e. type $\langle 1,2;1,1\rangle$);
\item $H$ is the intersection of three linear primes (i.e. type $\langle 1,1,1;1,1,1\rangle$);
\item $H$ is the intersection of a linear prime with a multiplicity two ideal primary to a linear prime (i.e. type $\langle 1,1;1,2\rangle$).
\end{enumerate}
\end{rmk}

\begin{prop}\label{e=3}
If $e(R/I)=3$, then $\pd(R/I)\leq 4$.
\end{prop}
\begin{proof}
By assumption $I^{un}$ has one of the five forms listed in Remark \ref{shapee=3}. In case (1), $I^{un}$ is Cohen-Macaulay by Proposition \ref{non-deg}(3), therefore Theorem~\ref{pd+1}(i) yields $\pd(R/I)\leq 3$. 
Case (2) is proved in \cite[Proposition~5.2]{MM}. Engheta proved Case (3) in \cite[Proposition~14, case 3]{En1} by showing that $I^{un}$ contains a quadric. 
Case (4) was proved by B. Engheta in \cite[Proposition~f, case 5]{En1}, however a few cases were missing in his analysis (e.g the cases where $I^{un}=(x,y)\cap (z,u)\cap (x+z,y+\alpha u)$ where $\alpha=0$ or $1$), thus we give a complete proof in Proposition \ref{111;111} below. Then we are left with case (5), i.e. $I^{un}=\p\cap K$, where $\p=(u,v)$ is a linear prime and $K$ is an ideal of multiplicity $2$ primary  to a linear prime $(x,y)$. If $K$ contains a linear form, then since $\p K \subseteq I^{un}$, it follows that $I^{un}$ contains a quadric; so $\pd(R/I)\leq 4$ by Theorem~\ref{pd+1}(iii). If $K$ does not contain a linear form, by \cite[Proposition~11]{En1} one has $K=(x,y)^2 + (ax+by)$ for homogeneous forms $a,b$ with ${\rm ht}(a,b,x,y)=4$. This case is proved in Proposition \ref{L11;12} below.
\end{proof}

First, we prove $\pd(R/I)\leq 4$ if either $I^{un}$ or $L=(f_1,f_2):f_3$ is the intersection of three linear primes. In particular, this proves case (4). 

\begin{prop}\label{111;111}
Let $K=\p_1\cap \p_2\cap \p_3$, where $\p_1,\p_2,\p_3$ are distinct linear primes of height $2$.
If either $K=I^{un}$ or $K=L$, then $\pd(R/I)\leq 4$. 
\end{prop}
\begin{proof}
Write $\p_1=(x,y)$, $\p_2=(z,w)$ and $\p_3=(u,v)$. If ${\rm ht}(\p_1+\p_2+\p_3)=6$, then $K=\p_1\p_2\p_3=(x,y)(z,w)(u,v)$ and $C=(xzv,ywu)$ is a complete intersection in $K$. 
Note that $L'=C:K=C+(xyzw,xyuv,zwuv)$, thus we cannot have $K=L$, since otherwise, by Theorem \ref{linkage}(d), $I^{un}$ contains only two linearly independent cubics, contradicting Remark \ref{observ}(a). If $K=I^{un}$, it is easily checked that the above squarefree monomial ideal $L'\sim I^{un}$ has $\pd(R/L')=3$, so $\pd(R/I)\leq 4$ by Theorem \ref{pd+1}(i). 

If ${\rm ht}(\p_1+\p_2+\p_3)\leq 4$ then $K$ is extended from a polynomial ring in at most 4 variables, thus Proposition \ref{<=t} yields $\pd(R/I)\leq 4$.

We may then assume ${\rm ht}(\p_1+\p_2+\p_3)=5$ and, without loss of generality, $v\in (x,y,z,w,u)$. If $v\in (x,y,u)$ then after possibly a change of coordinate we may assume $v\in (x,y)$, thus $vz\in K$ so the statement follows by Theorem \ref{pd+1}(iii). Analogously if $v\in (z,w,u)$. 

Assume then $v\notin (x,y,u)$ and $v\notin (z,w,u)$, then $v=x'+z'$ for $0\neq x'\in (x,y)$ and $0\neq z'\in (z,w)$. After possibly a change of coordinate, we may assume $v=x+z$, thus $K=(x,y)\cap (z,w)\cap (u,x+z)$. Similarly to the above, since $\h(x,y,z,w,u)=5$, it is easily checked that $L''=(xz(x+z), ywu):K$ has $\pd(R/L'')=3$ and contains only two linearly independent cubics. (another way to see this is by noticing that $v-(x+z)$ is regular on $R/L'$, where $L'$ is as above, and that $L''=L'R''$, where $R''=R/(v-x-z)\cong k[x,y,z,w,u]$.) 
Then, by Theorem \ref{pd+1}(i), one has $\pd(R/I)\leq 4$. 
\end{proof}

We now prove the remaining case of Proposition~\ref{e=3}.
\begin{prop}\label{L11;12}
If $I^{un}=\big[(x,y)^2 +(ax+by)\big]\cap (u,v)$ for linear forms $u,v,x,y$, and forms $a,b$ such that ${\rm ht}(a,b,x,y)=4$, then ${\rm pd}(R/I)\leq 4$.
\end{prop}

\begin{proof}
If ${\rm ht}(x,y,u,v)\leq 3$, then $I^{un}$ contains a quadric, so the statement follows by Theorem~\ref{pd+1}(iii). We may then assume ${\rm ht}(x,y,u,v)=4$. 
If ${\rm ht}(x,y,u,v,a,b)=6$, then $I^{un}=\big[(x,y)^2 + (ax+by)\big](u,v)$, which is linked to $L'=(x^2u, y^2v, xyuv, axuv - byuv, x^2y^2)$ via the complete intersection $x^2u,y^2v$. Since $\pd(R/L')=3$, we have that $\pd(R/I)\leq 4$.
We may then assume ${\rm ht}(x,y,u,v,a,b)\leq 5$. We first consider the case ${\rm ht}(x,y,u,v,a,b)=5$. In particular $ax+by\notin (x,y)^2+ (u,v)$. If   $\deg(a)=\deg(b)\geq 2$, then all cubics of $I^{un}$ lie in $(x,y)^2(u,v)$ by Lemma \ref{cubic}(i); thus $\pd(R/I)\leq 4$ by Remark \ref{observ}(c).
If instead $\deg(a)=\deg(b)=1$, we may assume $b\in (x,y,u,v,a)$. Since $(x,y)^2\subseteq (x,y)^2+(ax+by)$, we may clear the terms in $x$ and $y$ to further assume $b\in (u,v,a)$; now, after possibly choosing a different minimal generating set for $(x,y)$ and $(u,v)$ we may assume $b=u$. Since ${\rm ht}(x,y,u,v,a)=5$, we obtain
$$\begin{array}{ll}
I^{un} 	& = \big[(x,y)^2+(ax+uy)\big]\cap (u,v)\\
		& = (xav + yuv, y^2v, xyv, x^2v, xau + yu^2 , y^2u, xyu, x^2u).
\end{array}$$
Let $L'=(x^2u,y^2v):I^{un}=(y^2v, x^2u, xyuv, axuv - yu^2v, x^2y^2)$ and observe $\pd(R/L')=3$, then $\pd(R/I)\leq 4$ by Theorem \ref{pd+1}(i).

We may therefore assume ${\rm ht}(x,y,u,v,a,b)= 4={\rm ht}(x,y,u,v)$, i.e. $a,b\in (x,y,u,v)$. Similarly to the above, since $(x,y)^2\subseteq (x,y)^2+(ax+by)$, we may further assume $a,b\in (u,v)$. Thus  $ax+by \in (u,v)$, whence 
$$I^{un}=(ax+by)+\big[(x,y)^2\cap (u,v)\big]=(ax+by) + (x,y)^2(u,v).$$ 

Consider $L'=(x^2u,y^2v):I^{un}$, and note that $L'$ is an unmixed ideal with $e(R/L')=6$, by Theorem \ref{linkage}(b). 
It is easy to check that the ideal $L_1=\big(x^2u,y^2v,xyuv,x^2y^2,(ax-by)uv\big)$ is contained in $L'$. Also, one has
$$L_1=(x^2,v)\cap (y^2,u)\cap \big[(x,y)^2+(ax-by)\big],$$
so  $L_1$ is unmixed of multiplicity $6$ by Proposition \ref{Ass}. Therefore, $L'=L_1$.
Now observe that $L':uv=(x,y)^2+(ax-by)$, so by \cite[Prop.~11]{En1} $L':uv$ is $(x,y)$-primary with $\pd(R/(L':uv))=3$. Also, $L'+(uv)=(uv,x^2u,y^2v,x^2y^2)$ hence $\pd(R/(L'+(uv)))=3$. Then $\pd(R/L')\leq 3$ by Lemma \ref{ses}(1), and Theorem \ref{pd+1}(i) gives $\pd(R/I)\leq 4$.
\end{proof}

\section{The case $e(R/I)=6$}

In this short section we address the case $e(R/I)=6$.

\begin{prop}\label{e6}
If $e(R/I)=6$, then $\pd(R/I)\leq 5$.
\end{prop}
\begin{proof}
By Theorem~\ref{pd+1}(ii)--(iii) we may assume $I^{un}$ is generated in degree 3 or higher. Recall that $L$ is linked to $I^{un}$ via a complete intersection of two cubics in $I$, and so $e(R/L)=3$ by Theorem \ref{linkage}(b); thus, $L$ has one of the five forms listed in Remark \ref{shapee=3}.

In case (1), Proposition \ref{non-deg}(3) implies $L$ is Cohen-Macaulay, so $\pd(R/I)\leq 3$ by Theorem~\ref{pd+1}(i).  Cases (3) and (5) are proved in \cite[pp. 70--71]{En3}. Case (4) was proved in Proposition \ref{111;111} and case (2) in \cite[Proposition 5.3]{MM}.

\end{proof}

\section{The Cases $e(R/I) = 4$ or $5$}

These final two cases will occupy the remainder of the paper.  Given the intricacies of the arguments, each case  for the type of $I^{un}$ or $L := (f_1,f_2):I$ not covered by previous work has its own designated subsection.
  
\begin{prop}\label{e45}
If $e(R/I)=4$ or $5$, then $\pd(R/I)\leq 5$.
\end{prop}

\begin{proof} If $e(R/I) = 4$, set $K = I^{un}$.  If $e(R/I) = 5$, set $K = (f_1,f_2):I$.   Then $K$ is unmixed and $e(R/K)=4$ by Theorem \ref{linkage}(b).  By the associativity formula $K$ has one of the following $11$ types:
\begin{enumerate}
\item $\langle 4;1 \rangle$
\item $\langle 2;2 \rangle$
\item $\langle 1;4 \rangle$
\item $\langle 1,3;1,1 \rangle$
\item $\langle 1,1;1,3 \rangle$
\item $\langle 2,2;1,1 \rangle$
\item $\langle 1,2;2,1 \rangle$
\item $\langle 1,1;2,2 \rangle$
\item $\langle 1,1,2;1,1,1 \rangle$
\item $\langle 1,1,1;1,1,2 \rangle$
\item $\langle 1,1,1,1;1,1,1,1 \rangle$
\end{enumerate}
In case (1), $K$ is prime of almost minimal multiplicity; when $K$ contains a linear form, Theorem~\ref{pd+1}(b) yields $\pd(R/I) \le 3$; when $K$ is non-degenerate, by \cite{BS} either $\pd(R/K)\leq 4$, or $K$ contains a quadric, or is extended from a five-variable polynomial ring; therefore $\pd(R/I) \le 5$. 

In case (2), $K$ is $\p$-primary, where $\p = (z,q)$ is a prime ideal generated by a linear form $z$ and a quadric $q$.  Since $e(R_\p/K_\p) = 2$, we have $\p^2 \subseteq K$.  In particular, the quadric $z^2 \in K$ and hence $\pd(R/I) \le 4$ by Theorem~\ref{pd+1}.

Case (3) is proved in \cite[Prop. ~5.4 and 5.5]{MM}. In case (6), $K$ contains a quadric because  $K$ is the intersection of two primes each 
containing a linear form.  Then $\pd(R/I)\leq 4$ by Theorem \ref{pd+1}(iii). The remaining cases are proved in Propositions~\ref{L13;11}--\ref{13;11} (Case 4), \ref{L11;13}--\ref{11;13} (Case 5), \ref{12;21} (Case 7), \ref{11;22} (Case 8), \ref{112;111} (Case 9), \ref{111;112} (Case 10), and \ref{1111;1111} (Case 11).
\end{proof}

\bigskip

\subsection{Type $\boldsymbol{\langle 1,2;2,1\rangle}$}

\begin{prop}\label{12;21}
Let $K$ be the intersection of a prime $H_1$ of multiplicity two with an ideal $H$ of multiplicity two that is primary to a linear prime $(x,y)$.
If either $K=I^{un}$ or $K =L$, then $\pd(R/I)\leq 5$.
\end{prop}
\begin{proof}
If $K$ contains a quadric, then $\pd(R/I)\leq 4$ by Theorem \ref{pd+1}(iii). So we may assume that $[K]_2=0$. 
Since $e(R/H_1) = \h(H_1) = 2$, by Proposition \ref{non-deg}(2) $H_1=(z,q)$ for some linear form $z$ and  some quadric $q$. By \cite[Proposition~11]{En1}, either $H=(x,y^2)$ or $H=(x,y)^2 + (ax+by)$ for forms $a,b$ such that ${\rm ht}(x,y,a,b)=4$. In the former case, $K$ would contain a quadric.  In the latter case the statement is proved in the following Proposition~\ref{inters}.
\end{proof}

Note that the following proposition also covers the case when $I^{un}$ or $L$ is of type $\langle 1,1,2;1,1,1\rangle$ and the intersection of the two linear primes has the form $(z,v) \cap (z,w) = (z,vw)$.  

\begin{prop}\label{inters}
Let $K=H\cap (z,q)$, where $z$ is a linear form and $q$ is a quadric with ${\rm ht}(z,q)=2$, and $H=(x,y)^2+(ax+by)$ for linear forms $x,y$ and homogeneous forms $a,b$ with ${\rm ht}(x,y,a,b)=4$. Further suppose that $K$ contains no quadrics. If either $K=I^{un}$ or $K = L$, then $\pd(R/I)\leq 5$. 
\end{prop}
We divide the proof into two cases, based on the degrees of $a$ and $b$. First, we consider the case where $a$ and $b$ are linear forms.

\begin{lem}\label{lin2}
Proposition~\ref{inters} holds when $\deg(a)=\deg(b)=1$. 
\end{lem}
\begin{proof}
Since $K$ is generated in degree 3 and higher, we have ${\rm ht}(x,y,z)=3$, and since by assumption ${\rm ht}(x,y,z,a,b)\geq 4$, without loss of generality we may assume $\h(x,y,z,a)=4$. \\

\noindent \underline{Case 1: Assume ${\rm ht}(x,y,z,a,b)=4$.}\\
Then $b\in (x,y,z,a)$, and since $(x,y)^2\subseteq H$, we may assume $b\in (z,a)$. After possibly a linear change of the $x,y$ variables, we may assume $b=z$. Then $K=\big[(x,y)^2+(ax+yz)\big]\cap (z,q)$. Next, we prove that ${\rm ht}(x,y,z,a,q)=4$. Indeed, if not, then $J:=\big[(x,y)^2+(ax,z)\big]:q=(x,y)^2+(ax,z)$ and by Remark \ref{observ}(b) all cubics of $K$ are contained in $(z)+qJ=(z) + q\big[(x,y)^2+(ax)\big]$.  Therefore, $[K]_3 \subseteq (z)$, contradicting Remark \ref{observ}(a). Therefore ${\rm ht}(x,y,z,a,q)=4$, that is, $q\in (x,y,z,a)$. After reducing $q$ modulo $z$ we may actually assume $q\in(x,y,a)$, say $q=l_1x+l_2y+l_3a$ for linear forms $l_1,l_2,l_3$. We prove this case in Lemma \ref{l_i} below.\\
 
\noindent \underline{Case 2: Assume ${\rm ht}(x,y,z,a,b)=5$.}\\
 The assumption yields that $H+(z)$ is an $(x,y,z)$-primary ideal. Observe that $q\in (x,y,z)$, since otherwise $\big[H+(z)\big]:q= H+(z)$ and
$$K\subseteq \big[H+(z)\big]\cap (z,q)=(z)+\big[\big(H+(z)\big)\cap (q)\big]=(z)+\big[H+(z)\big]q=(z) + qH$$
proving that every cubic in $K$ is multiple of $z$, a contradiction to Remark~\ref{observ}(a). Therefore $q\in (x,y,z)$, and after possibly clearing the term in $z$, we may assume $q\in (x,y)$. Thus $q(x,y)\subseteq K$, but $q\notin H$, since otherwise $K$ would contain a quadric.  We claim $K=q(x,y) + zH$. Indeed, by the above, $(q)\cap \big[H+(z)\big]=q\big[\big(H+(z)\big):q\big]=q(x,y,z)$, then
$$K\subseteq \big[H+(z)\big]\cap(z,q)=(z)+\big[(q)\cap \big(H+(z)\big)\big]=(z)+q(x,y,z)=(z)+q(x,y)$$
Since $(z)+q(x,y)\subseteq (z,q)$, it follows that \[K=H\cap \big[(z)+q(x,y)\big]=q(x,y)+\big[(z)\cap H\big]= q(x,y) + zH.\]

If $K=L$, note that $K:z=H:z=H$ whence $\pd(R/(K:z))=3$ (see e.g. \cite[Proposition~11]{En1}), and $K+(z)=q(x,y)+(z)$; hence $\pd(R/(K,z))=3$. By Lemma \ref{ses} we have $\pd(R/K)\leq 3$ and then $\pd(R/I)\leq 4$ by Theorem \ref{pd+1}(i). 

If instead $K=I^{un}$, write $q=l_1x+l_2y$ with ${\rm ht}(l_1,z)={\rm ht}(l_2,z)=2$. Note that $K=I^{un}$ is generated in degree at least 3 and is extended from $k[x,y,z,a,b,l_1,l_2]$, thus if ${\rm ht}(x,y,z,a,b,l_1,l_2)\leq 5$ the statement follows by Proposition \ref{<=t}. We may then assume ${\rm ht}(x,y,z,a,b,l_1)=6$. Then $C=(y^2z,xq)$ is a complete intersection of height two, and one checks
that $L'=C:K=(y^2z,x(l_1x+l_2y),xyl_1,xy^2,z(xal_1-y(bl_1-al_2)))$.  Moreover, $L':y = (xl_1, xy, yz)$ and $L'+(y) = (l_1x^2, al_2xz, y)$.  Since $\pd(R/(L':y)) = 2$ and $\pd(R/(L' + (y))) = 3$, by Lemma~\ref{ses} we have $\pd(R/L')\le 3$.  Therefore $\pd(R/I)\leq 4$ by Theorem \ref{pd+1}.
\end{proof}

The following lemma concludes the proof of Lemma~\ref{lin2}.
\begin{lem}\label{l_i}
Proposition~\ref{inters} holds if in addition one assumes that $q=l_1x+l_2y+l_3a$, $b=z$, and $a,l_1,l_2,l_3$ are linear forms.
\end{lem}
\begin{proof} 
If ${\rm ht}(x,y,z,a,l_1,l_2,l_3)\leq 5$, the statement follows by Proposition \ref{<=t}. We may then assume $\h(x,y,z,a,l_1,l_2,l_3)\geq 6$. 
Since $\h(x,y,z,a)=4$, it follows that $H+(z)=(x,y)^2+(ax,z)=(x,y^2,z)\cap (x^2,y,z,a)$, so $\pd(R/(H+(z)))=4$. We claim that $L':=(H+(z)):q = (x,y^2,z):q$.
Observe that $L'=((x,^2,z):q) \cap ((x^2,y,z,a):q)$. We first show that $(x,y^2,z):q$ is a proper ideal. If not, then $l_2y+l_3a \in (x,y^2,z)$. Then $l_3a\in (x,y,z)$ and $l_2y\in (x,y^2,z,a)$. Since $\h(x,y,z,a)=4$ we obtain $l_2,l_3\in (x,y,z,a)$ and thus $\h(x,y,z,a,l_1,l_2,l_3)\leq 5$, which is a contradiction.

Now, since $l_2y+l_3a\in (x^2,y,z,a)$, we have  $ (x^2,y,z,a):q = (x^2,y,z,a):l_1x\supseteq (x,y,z,a)$; since $(x,y^2,z):q$ is a proper ideal, then it is contained in $(x^2,y,z,a):q$. This fact and the above decomposition of $L'$ prove the claim.
Since $L'=(x,y^2,z):q$ is either $(x,y,z)$ or $(x,y^2,z)$, then $\pd(R/L')=3$.

Now, if $K=L\sim I^{un}$, we have the short exact sequence
\[ 
0 \longrightarrow R/L'\stackrel{\cdot q}{\longrightarrow} R/(H+(z)) \longrightarrow R/(H+(z,q))\longrightarrow 0
\]
By the above,  $\pd(R/(H+(z)))=4$ and $\pd(R/L')=3$, so Lemma \ref{ses}(1) implies $\pd(R/(H+(z,q)))\leq 4$, and then Lemma \ref{ses}(2) applied to 
\[ 
0 \longrightarrow R/K \longrightarrow R/H \oplus R/(z,q) \longrightarrow R/(H+(z,q))\longrightarrow 0
\]
yields $\pd(R/K)\leq 3$. Then $\pd(R/I)\leq 4$ by Theorem \ref{pd+1}(i).

We may then assume $K=I^{un}$. 
When $l_3\in (x,y,z)$, we have $K=(z,l_1x+l_2y)\cap ((x,y)^2+(ax+yz))$ and $\h(x,y,z,a,l_1,l_2)=6$, so one checks that $\pd(R/L'')=3$, for the ideal $L''=(l_1x^2+l_2xy,zy^2):K$ linked to $K=I^{un}$. Then $\pd(R/I)\leq 4$ by Theorem \ref{pd+1}(i). 

We may then assume $l_3\notin (x,y,z)$. By Remark \ref{observ}(b), $I^{un} = H \cap (z, qL')=H\cap (z,q(x,y^2))$. Since $qy^2\in H$, we have that $(z,qx)=(z,c)$, where $c:=qx+l_3yz\in H$. 
Since $H\cap (z)=zH$,  
\[ I^{un}=(qy^2, c)+H\cap (z)=(qy^2, c)+zH.\]
Since $\deg(qy^2)=4$, we have the containment $I\subseteq (c)+zH$ and so, after possibly taking linear combinations of the generators of $I$, at least two generators of $I$ are multiple of $x$. Then $\pd(R/I)\leq 4$ by Corollary \ref{x}.

\end{proof}

Next, we consider the case where $H=(x,y)^2+(ax+by)$, where $a$ and $b$ are forms of degree at least 2.
\begin{prop}\label{deg2}
Proposition~\ref{inters} holds true when $\deg(a)=\deg(b)\geq 2$. 
\end{prop}
\begin{proof} 
Let $h=ax+by$. Since we may assume that $K$ is generated in degree three and higher, we have in particular that ${\rm ht}(x,y,z)=3$ and $q\notin H+(z)$.

Assume $K=L\sim I^{un}$. First, we note that $L:z= H$ so $\pd(R/(L:z))=3$. By Lemma \ref{ses}(1) it then suffices to prove $\pd(R/L+(z))\leq 5$. Set $L'=L+(z)$; by Remark \ref{observ}(f), 
$$
L':= L+(z)= (z,q)+(H+(z))=(z) + (q) \cap (H+(z))= (z) + q((H+(z)):q)
$$
Then $L'+(q)=(z,q)$ so $\pd(R/L'+(q))=2$. On the other hand, using Remark \ref{observ}(e) one observes that $L':q = ((H+(z)):q) + (z)$. Since $z\in (H+(z)):q$, then $L':q = (H+(z)):q$, which is a proper ideal because $q\notin H+(z)$; if $\h(L':q)=3$, since $(x,y,z)\subseteq L':q$, then $L':q=(x,y,z)$ and so $\pd(R/L':q)=3$. If $\h(L':q)>3$, we claim that $L':q=(x,y,z,F)$ for some form $F\notin (x,y,z)$. Indeed, since $z$ is regular on $R/H$, then $\pd(R/H+(z))= 4$, so for instance by \cite[Lemma~2.6]{HMMS4} we deduce that the largest height of an associated prime of $H+(z)$ is at most 4, and thus the same is true for $L':q$, proving that $L':q$ is unmixed of height 4. Since $(x,y,z)\subseteq L':q$, then $L':q/(x,y,z)$ is an unmixed ideal of height 1 in the polynomial ring $R/(x,y,z)$, and therefore is principal. This proves $L':q=(x,y,z,F)$ for some $F\in R$. In this case $\pd(R/L':q))=4$.

In either case we have $\pd(R/L':q)\leq 4$ and then by Lemma \ref{ses}(1) one obtains $\pd(R/L')\leq 4$, which concludes the proof when $K=L\sim I^{un}$. For the rest of the proof, we then assume $K=I^{un}$. 

If $\deg(h)=3$ and $h\notin H':=(x,y)^2 +(z,q)$, or if $\deg(h)\geq 4$, then $[K]_3= \big[(x,y)^2\cap (z,q)\big]_3$.  When $[K]_3=\big[(x,y)^2\cap (z,q)\big]_3$, we must have ${\rm ht}(x,y,z,q)\leq 3$, since otherwise $\big[(x,y)^2\cap (z,q)\big]_3=z(x,y)^2 \subseteq (z)$, which contradicts Remark \ref{observ}(a). 
Also, since ${\rm ht}(x,y,z)=3$ it follows that $q\in (x,y,z)$. Clearing terms in $z$, we may assume $q=l_1x+l_2y$ for linear forms $l_1,l_2$. Then all cubics in $K$ can be written in terms of $x,y,z,l_1,l_2$ and Remark \ref{observ}(c) proves the statement.

We may then assume $\deg(h)=3$ and  $h\in (x,y)^2 +(z,q)$. By Lemma \ref{cubic}(ii), after possibly modifying $a$ and $b$ modulo $x$ and $y$ we can actually assume $h\in (z,q)$, so $K=(h)+\left((x,y)^2\cap (z,q)\right)$.\\

\noindent \underline{Case 1: Assume $\h(x,y,z,q)=4$.}\\
 Then $K=(h)+(x,y)^2(z,q)$ and so $[K]_3=(h)+ z(x,y)^2$; then, after possibly taking linear combinations, we may assume two of the three generators of $I$ are multiple of $z$. The statement now follows by Corollary \ref{x}.\\

\noindent \underline{Case 2: Assume ${\rm ht}(x,y,z,q)=3$.}\\
 This implies that $q\in (x,y,z)$. After reducing $q$ modulo $z$ we may assume $q=l_1x+l_2y\in(x,y)$.  Recall that if $\h(x,y,z,l_1,l_2) = 3$, then $K$ contains a quadric.  Hence we may assume $\h(x,y,z,l_1)=4$. Then by Remark \ref{observ}(b) we have
\[K=(h) + q(x,y) + z(x,y)^2\]

Since $h=ax+by\in (z,q)$ and $q\in (x,y)$ and $\h(x,y,z)=3$, we have $h\in z(x,y)+(q)$, and we can write $h=c_1xz+c_2yz+c_3q$ for linear forms $c_1,c_2,c_3$. Moreover, since $q(x,y)+z(x,y)^2\subseteq K$, for every $i=1,2,3$ we may assume either $c_i=0$ or else $c_i\notin (x,y)$. Then 
$$K=((c_1x+c_2y)z+c_3q) + q(x,y) + z(x,y)^2,$$
with $q=l_1x+l_2y$ and $\h(x,y,z,l_1,l_2)\geq 4$, and if $c_i\neq 0$ then $\h(x,y,c_i)=3$.\\

\noindent \underline{Case 2a: Assume ${\rm ht}(x,y,z,l_1,l_2)=5$.}\\ 
The height condition implies that $C=(y^2z,x(l_1x+l_2y))$ is a complete intersection and $(z,q)=(z,l_1x+l_2y)$ is a prime ideal.
Then the associated primes of $C$ are $\p_1=(x,y)$, $\p_2=(y,l_1)$, $\p_3=(z,x)$, and $\p_4=(z,q)$. Set $L=C:K$. Also, one can check that $L_{\p_1}=(C:K)_{\p_1}=(y^2,q)_{\p_1}$, $L_{\p_2}=(C:K)_{\p_2}=(y^2,l_1)_{\p_2}$, $L_{\p_3}=(C:K)_{\p_3}=(z,x)_{\p_3}$, and $L_{\p_4}=(C:K)_{\p_4}=R_{\p_4}$; in particular, ${\rm Ass}(R/L)=\{\p_1,\p_2,\p_3\}$. 
Also, $L$ contains $L'=(y^2z, l_1xy, l_1x^2+l_2xy, c_3l_1^2x+c_1l_1xz-c_2l_1yz+c_1l_2yz)$.
One can see that $L'$ is unmixed if and only if ${\rm ht}(x,y,I_2(M))=4$, where $M=\left(\begin{array}{ccc}
c_3& c_2 & c_1\\
-z & l_2 & l_1 \end{array}\right)$, which is equivalent to $K$ being unmixed. Because $K$ is unmixed, the ideal $L'$ is unmixed and $\pd(R/L')=3$. It is easily checked that ${\rm Ass}(R/L')={\rm Ass}(R/L)=\{\p_1,\p_2,\p_3\}$ and $L_{\p_i}=L'_{\p_i}$ for all $i=1,2,3$. Therefore, $L=L'$, thus $\pd(R/L)=3$, and then $\pd(R/I)\leq 3+1=4$ by Theorem \ref{pd+1}(i).\\

\noindent \underline{Case 2b: Assume ${\rm ht}(x,y,z,l_1,l_2)=4$.}\\
We may assume $l_2\in (x,y)$. After possibly a linear change of coordinates we may further assume 
 that either  $q=l_1x$ or $q=l_1x+y^2.$

Also, since ${\rm ht}(x,y,z,l_1)=4$, the ideal $C=(xq,y^2z)$ is a complete intersection in $K=I^{un}$. 
Let $L_1=C:K$. By Theorem \ref{linkage}(b), $e(R/L)=5$ and $L_1\supseteq L'=(x^2l_1,xyl_1, xy^2,y^2z, l_1h)$, which is also an ideal of multiplicity $5$. Looking at the generic resolution of $R/L'$ one sees that $L'$ is unmixed if and only if ${\rm ht}(x,y, c_2l_1, c_2z, c_1z+c_3l_1)\geq 4$ if and only if (since ${\rm ht}(x,y,z,l_1)=4$) ${\rm ht}(x,y,c_2)=3$, in which case we have $L_1=L'$, and then $\pd(R/L_1)=3$, which yields $\pd(R/I)\leq 3+1=4$ by Theorem \ref{pd+1}(i).

If, instead, ${\rm ht}(x,y,c_2l_1, c_2z, c_1z+c_3l_1)\leq 3$, then $c_2\in (x,y)$ and, by the above, we may assume $c_2=0$. 
Now, if $\h(x,y,z,l_1,c_1,c_3)\leq 5$, the statement follows by Proposition \ref{<=t}. Assume then $\h(x,y,z,l_1,c_1,c_3)=6$.
Take $C=(xc_1z+c_3q, zy^2)$, the height assumption implies that $C$ is a complete intersection, and observe that $L'=(c_1z+c_3l_1, y^2z, y^2c_3)\subseteq L_1=C:K$. Since $L'=(c_1z+c_3l_1, y^2c_3): (c_1,c_3)$, then $L'$ is directly linked to a complete intersection, so by Theorem \ref{linkage}(a) and (b) the ideal  $L'$ is Cohen-Macaulay and $e(R/L')=6-1=5=e(R/L_1)$. Then $L'=L_1$, and so $\pd(R/I)\leq 3$ by Theorem \ref{pd+1}(i).

\end{proof}

\bigskip

\subsection{Type $\boldsymbol{\langle 1,1;2,2\rangle}$}

In this subsection we prove the following result: 
\begin{prop}\label{11;22}
Let $K=K_1\cap K_2$ where $K_1$ and $K_2$ are primary to two distinct linear primes and $e(R/K_i)=2$ for $i=1,2$. If either $K= I^{un}$ or $K=L$, then $\pd(R/I)\leq 5$.
\end{prop}

We first observe that in most cases the statement follows by our previous work.
\begin{lem}
Proposition~\ref{11;22} holds if either $K_1$ or $K_2$ contains a linear form, or if $K$ contains a quadric. 
\end{lem}
\begin{proof} If $K$ contains a quadric the statement follows by Theorem \ref{pd+1}(iii). Assume then $K_2$ contains a linear form. Then, by \cite[Proposition~11]{En1}, we can write $K_2=(u,v^2)$ for some $u,v \in R_1$, and there are $x,y\in R_1$ such that either $K_1=(x,y^2)$, or $K_1=(x,y)^2+(ax+by)$ for forms $a,b$ with ${\rm ht}(a,b,x,y)=4$. In the former case, $K$ contains a quadric, hence $\pd(R/I)\leq 4$ by Theorem~\ref{pd+1}(iii). In the latter case the statement follows by Proposition~\ref{inters}. 
\end{proof}

We may then assume $K$ contains no quadrics. Moreover, by \cite[Proposition~11]{En1}, if $K_2$ does not contain a linear form, there are forms $c,d$ such that $K_2=(u,v)^2 +(cu+dv)$. Hence, we need to prove the following.
\begin{prop}\label{1122}
Proposition~\ref{11;22} holds when $K$ contains no quadrics and $K_1=(x,y)^2+(ax+by)$ and $K_2=(u,v)^2+(cu+dv)$, where 
${\rm ht}(x,y,a,b)={\rm ht}(u,v,c,d)=4$.
\end{prop}

The proof follows by Propositions \ref{1122deg1} and \ref{1122deg2} below.

\begin{lem}\label{1122lem}
Assume $K$ is as in Proposition \ref{1122}. Then
\begin{enumerate}
\item At least one cubic of $K$ is not contained in $H=(x,y)^2 \cap K_2$;
\item One has $\h(x,y,u,v,a,b)\leq 5$ and $\h(x,y,u,v,c,d)\leq 5$;
\item If $\deg(a)=\deg(b)\geq 2$, then $\deg(a)=\deg(b)=2$ and $\h(x,y,u,v,a,b)\leq 4$.
\end{enumerate}
\end{lem}

\begin{proof}
Since $K$ contains no quadrics, $\h(x,y,u,v)=4$ and so the ideal $(x,y)^2\cap (u,v)^2=(x,y)^2(u,v)^2$ contains no cubics.

(1) It suffices to prove that if $[K]_3=[H]_3$, then $\h\left([K]_3\right)\leq 1$, which is a contradiction, by Remark \ref{observ}(a).  If $c,d \in R_2$, then $cu+dv \in R_3$, so by Lemma \ref{cubic}(i) there is a cubic $F$ with $[H]_3 \subseteq \left[ (F) + \left((x,y)^2 \cap (u,v)^2\right)\right]_3 = (F)$. Similarly if $\deg(c) = \deg(d) \ge 3$, there are no cubics in $H$, a contradiction.
Assume then $c,d\in R_1$. If $\h(x,y,u,v,c,d)=4$, then since $(u,v)^2 \subseteq K_2$, we may further assume $(c,d)=(x,y)$ and thus $H=(x,y)^2 \cap \left((u,v)^2 + (xu+yv)\right)=(xu+yv)(x,y) + \left((x,y)^2 \cap (u,v)^2\right)$. Therefore all cubics in $H$ are multiple of $xu+yv$.

Finally, we may assume $\h(x,y,u,v,c)=5$. In particular, this implies $cu+dv\notin (x,y)^2+(u,v)^2$ and since $(u,v)\subseteq J:=\left((x,y)^2+(u,v)^2\right) :(cu+dv)$ and $J$ is $(x,y,u,v)$-primary, we deduce that $(u,v)+(x,y)^2 \subseteq J \subseteq (u,v,x,y)$. If a linear form $\ell$ of $(x,y)$ lies in $J$, then $\ell(cu+dv)\in (x,y)^2+(u,v)^2$ yields $\ell c u \in (x,y)^2+(u^2,v)$, so $\ell c \in (x,y)^2 + (u,v)$, and then $c\in (x,y,u,v)$, which is a contradiction. Then $J=(x,y)^2 + (u,v)$, and thus
\[ H = (x,y)^2 \cap \left[ (u,v)^2 + (cu+dv)J \right] = (cu+dv)(x,y)^2 + \left((x,y)^2\cap (u,v)^2\right).\]
Since $(x,y)^2\cap (u,v)^2$ contains no cubics, $[K]_3\subseteq (cu+dv)$.

(2) Assume toward a contradiction that $\h(x,y,u,v,a,b)=6$.  Then $K\subseteq K_1(u,v) \cap K_2$. By part (1) it suffices to prove that all cubics in $K$ are contained in $(x,y)^2 \cap K_2$. 
If $\deg(a)=\deg(b)\geq 2$ this follows because all cubics in $K_1\cap (u,v)=K_1(u,v)$ are contained in $(x,y)^2(u,v)\subseteq (x,y)^2$.
Assume then $a,b\in R_1$. If there was a cubic of the form $\ell (ax+by) + F $ in $K$, for some $F\in (x,y)^2(u,v)$ and $0\neq \ell \in (u,v)$, then without loss of generality we may assume $\ell=u$. 

Then $(ax+by)u \in K'=K_2+(x,y)^2(u,v)=(u,v)(u,v,x^2, xy,y^2) + (cu+dv)$ and so $ax+by\in K':u$. Notice that 
\begin{align*}
K':u 	& = \left((x,y)^2+(u,v) \right) + (vx^2, vxy, vy^2, v^2, cu+dv):u \\
	& \subseteq \left((x,y)^2+(u,v) \right) + (x^2, xy, y^2, v^2, cu+dv):u\\
	& \subseteq (x,y)^2+(u,v).
\end{align*}
The last inclusion holds because $\h(u,v,c,d)=4$ so $u\notin A:=(x^2, xy, y^2, v^2, cu+dv)_{\p}$, where $\p=(x,y,u,v)$, because the ideal $(x,y)^2+(u,v)$ is $\p$-primary, and because $A= \left((x,y)^2 + (K_2)\right)_{\p} $, so $(K':u)_{\p}=\left((x,y)^2+ (u,v)\right)_{\p}$.

Then $ax+by\in (x,y)^2+(u,v)$ and thus $ax \in (x^2,y,u,v)$, so $a\in (x,y,u,v)$, a contradiction to our assumption. 


(3) By part (1) it suffices to show that if $\deg(a)=\deg(b)\geq 3$ or $\h(x,y,u,v,a,b)\geq 5$, then $[K]_3=\big[(x,y)^2\cap ((u,v)^2+(cu+dv))  \big]_3$.
This is clear if $\deg(a)=\deg(b)\geq 3$, because in this case $\deg(ax+by)\geq 4$. On the other hand, if $\h(x,y,u,v,a,b)\geq 5$, then $ax+by\notin (x,y)^2+(u,v)$, and Lemma \ref{cubic}(i) implies that $[K]_3=\big[(x,y)^2\cap ((u,v)^2+(cu+dv))  \big]_3$.
\end{proof}

%

\begin{lem}\label{1122deg1}
Proposition \ref{1122} holds if one further assumes $a,b,c,d$ are linear forms.
\end{lem}
\begin{proof}
Since $K$ contains no quadrics, ${\rm ht}(x,y,u,v)=4$, and by Lemma \ref{1122lem}(2) we have ${\rm ht}(x,y,u,v,a,b)\leq 5$ and  ${\rm ht}(x,y,u,v,c,d)\leq 5$. 

If ${\rm ht}(x,y,u,v,a,b,c,d)\leq 5$ the statement follows by Proposition \ref{<=t}. 
We may then assume ${\rm ht}(x,y,u,v,a,b)={\rm ht}(x,y,u,v,c,d)=5$, and $b\in (x,y,u,v,a)$ and $d\in(x,y,u,v,c)$. Since $(x,y)^2\subseteq K_1$ and $(u,v)^2\subseteq K_2$, without loss of generality, we may assume $b\in (u,v,a)$ and $d\in (x,y,c)$. After possibly a linear change of the $x$ and $u$ variables, we may further assume $b\in(u,v)$ and $d\in (x,y)$ and then take $b=u$ and $d=x$. 
Observe that $h={\rm ht}(x,y,u,v,a,c)\leq 6$. We prove $h\leq 5$, which concludes the proof (by Proposition \ref{<=t}).

Indeed, if $h=6$, then $J:=\left((x,y)^2+(u,v)^2+(cu+xv)\right):(ax+uy)=(x,y)^2+(u,v)^2 +(cu)$, so $(ax+uy)J$ contains no cubics. Since Remark \ref{observ}(b) yields that 
$K=\left((x,y)^2 + (ax+uy)J\right) \cap K_2$, then 
$[K]_3\subseteq [(x,y)^2 \cap K_2]_3$, 
which is impossible by Lemma \ref{1122lem}(1).
\end{proof}

The next lemma finishes the proof of Proposition~\ref{11;22}.
\begin{lem}\label{1122deg2}
Proposition \ref{1122} holds when $\deg(a)=\deg(b)\geq 2$.
\end{lem}

\begin{proof}
Since $K$ contains no quadrics, then ${\rm ht}(x,y,u,v)=4$. By Lemma \ref{1122lem}(3) we only need to prove the statement when ${\rm ht}(x,y,u,v,a,b) = 4$ and $\deg(a)=\deg(b)=2$.  

We first show ${\rm ht}(x,y,u,v,c,d)=4$. If not, we may assume ${\rm ht}(x,y,u,v,c)=5$. We claim that $(x,y)^2\cap K_2= (x,y)^2K_2$. It suffices to prove $(x,y)^2\cap K_2 \subseteq (x,y)^2K_2$. By Remark \ref{observ}(b) we have $(x,y)^2\cap K_2=(x,y)^2\cap [(u,v)^2+(cu+dv)J]$, where $J=\big[(x,y)^2+(u,v)^2\big]:cu+dv\subseteq \big[(x,y)^2+(u^2,v)\big]:cu=\big((x,y)^2,u,v\big)$.  Therefore  
\begin{align*}
(x,y)^2\cap K_2 \phantom{\,\,\,\,\,} &
=(x,y)^2 \cap \big[(u,v)^2+(cu+dv)J\big]\\
&\subseteq (x,y)^2\cap \big[(u,v)^2+\big(cu+dv\big)\big((x,y)^2+(u,v)\big)\big]\\ 
&=(x,y)^2(cu+dv)+(x,y)^2\cap (u,v)^2\\
&=(x,y)^2K_2
\end{align*}
Since $ax+by\in R_3$, by Lemma \ref{cubic}(i), there exists $F\in R_3$ such that  
$[K]_3\subseteq \big[(ax+by+F)+ (x,y)^2\cap K_2\big]_3 \subseteq \big[(ax+by+F)+ (x,y)^2K_2\big]_3=\big[(ax+by+F)\big]_3$. This contradicts Remark~\ref{observ}(a).

Therefore, we may assume ${\rm ht}(x,y,u,v,c,d)=4={\rm ht}(x,y,u,v,a,b)$.
Since $\h(x,y,u,v) =4$, we have $a,b\in (x,y,u,v)$. Since $(x,y)^2 \subseteq K_1$, we may further assume $a,b\in(u,v)$. Similarly we may assume $c,d\in (x,y)$.

(1) If $\deg(c)=\deg(d)=1$, since $c,d\in (x,y)$ and ${\rm ht}(u,v,c,d)= 4$, we have $(c,d)=(x,y)$. 
Then $K_2\cap (x,y)=(cu+dv)+ \left((u,v)^2 \cap (x,y)\right) = (u,v)^2(x,y) + (cu+dv)$. Since $K=K_1\cap K_2\cap (x,y)$, then 
$K= K_1\cap \left((u,v)^2(x,y) + (cu+dv)\right).$

Lemma \ref{1122lem}(1) and Remark \ref{observ}(b) imply that $ax+by\in (x,y)^2+(u,v)^2(x,y)+(cu+dv)$, and by Lemma \ref{cubic}(ii) we may assume $ax+by\in (u,v)^2(x,y)+(cu+dv)$. Then there exists $\ell \in R_1$ such that $ax+by=G+\ell(cu+dv)$ for some $G\in (u,v)^2(x,y)$. In particular, $ax+by\in S=k[x,y,u,v,\ell]$. Then both $K_1$ and $K_2$ are extended from $S$, and so is $K$. The statement now follows from Proposition \ref{<=t}.

(2) We may then assume $\deg(c)=\deg(d)\geq 2$. By Lemma \ref{1122lem}(3) we may assume $\deg(c)=\deg(d)=2$ and ${\rm ht}(x,y,u,v)={\rm ht}(x,y,u,v,c,d)={\rm ht}(x,y,u,v,a,b)=4$. Similarly to the above, intersecting $K$ with $(x,y)$ and $(u,v)$ one obtains $K=\big[(x,y)^2(u,v)+(ax+by)\big]\cap \big[(u,v)^2(x,y)+(cu+dv)\big]$.
As before, Lemma \ref{1122lem}(1) yields $ax+by\in (x,y)^2(u,v)+(u,v)^2(x,y)+(cu+dv)$, and thus a proof similar to the one of Lemma \ref{cubic}(ii) yields 
$$K=(ax+by)+\big[(x,y)^2(u,v)\cap \big((u,v)^2(x,y)+(cu+dv)\big)\big].$$
By Remark \ref{observ}(b), 
 if $cu+dv\notin H=(x,y)^2(u,v)+(x,y)(u,v)^2$ then $[K]_3=[(ax+by)]_3$, which contradicts Remark \ref{observ}(a). Thus $cu+dv\in H$ and then $cu+dv$ is written purely in terms of $u,v,x,y$, i.e. $cu+dv\in A_3$ where $A=k[x,y,u,v]$. By the above, $ax+by\in \big[H+(cu+dv)\big]_3\subseteq A_3$, thus $K_1$ and $K_2$ are extended from $A$, and so is $K$. Therefore $\pd(R/I)\leq 4$ by Proposition \ref{<=t}. 
\end{proof}
\bigskip

\subsection{Type $\boldsymbol{\langle 1,1,2;1,1,1\rangle}$}

Note that in addition to the case at hand, the following proposition covers some cases of type $\langle 1,1,1,1;1,1,1,1\rangle$, that is, certain intersections of four linear primes of height two.

\begin{prop}\label{112;111}
Let $K=(x,y)\cap (u,v)\cap (z,q)$ for some linear forms $x,y,u,v,z$ and a quadric $q$. If either $K =  I^{un}$ or $K=L$, then  
$\pd(R/I)\leq 5$. 
\end{prop}
\begin{proof}
If ${\rm ht}(x,y,u,v)\le 3$, or ${\rm ht}(x,y,z)\le 2$, or ${\rm ht}(u,v,z)\le2$, then $K$ contains a quadric, and by Theorem~\ref{pd+1}(iii) we have $\pd(R/I)\leq 4$. We may then assume ${\rm ht}(x,y,u,v)=4$ and ${\rm ht}(u,v,z)={\rm ht}(x,y,z) = 3$. 

Observe that  ${\rm ht}(x,y,u,v,z,q)\leq 5$, since otherwise $K=(x,y)(u,v)(z,q)$ and all cubics in $K$ would be multiples of $z$, contradicting Remark~\ref{observ}(a). Write $K_1=(x,y)\cap (u,v)$ and $K_2=(z,q)$.\\

\noindent \underline{Case 1: Assume ${\rm ht}(x,y,u,v,z)=5$.}\\ Then $q\in (x,y,u,v,z)$. We first show that either $q\in (x,y,z)$ or $q\in(u,v,z)$. Indeed, if not, by Remark \ref{observ}(b) we have $K_1\cap (z,q)=K_1\cap \big[(z)+qJ\big]$, where $J=\big[K_1+(z)\big]:q$. The height condition yields that $K_1+(z)=(x,y,z)\cap (u,v,z)$, and then by assumption $q$ is regular on $R/K_1$, whence $\big[K_1+(z)\big]:q=K_1+(z)$. 
 Then $$K=K_1\cap \big[(z)+q\big(K_1+(z)\big)\big]=K_1\cap \big[(z)+qK_1\big]=qK_1+\big(K_1\cap (z)\big).$$ 
 Since $qK_1$ is generated in degree four, all cubics in $K$ are multiples of $z$, which contradicts Remark~\ref{observ}(a).

We may then assume $q\in (x,y,z)$. If also $q\in (u,v,z)$, then $q\in K_1 +(z)$, and then $K$ contains a quadric, so $\pd(R/I)\leq 4$ by Theorem \ref{pd+1}(iii). So we may assume $\h(u,v,z,q)=4$. Also, since $K_2=(z,q)$, without loss of generality, we may assume $q\in (x,y)$, say $q=l_1x+l_2y$, for $l_1,l_2\in R_1$.  
Then $K=(x,y)\cap K'$, where $K'=(u,v)(z,q)$, so the modularity law and $\h(x,y,u,v,z)=5$ imply $K=(qu,qv,xzu,xzv,yzu,yzv)$.
 
If $K=L\sim I^{un}$, 
since $\pd(R/(K:z))=\pd(R/(x,y)\cap(u,v))=3$ and $\pd(R/(K+(z)))=\pd(R/(z,q(u,v)))=3$, then by Lemma \ref{ses} we have $\pd(R/K)\leq 3$, hence $\pd(R/I)\leq 4$ by Theorem \ref{pd+1}(i). We may then assume $K=I^{un}$.

Observe that if $\h(x,y,u,v,z,l_1,l_2)\leq 5$, then the statement follows by Proposition \ref{<=t}. We may then assume $\h(x,y,u,v,z,l_2)=6$. This implies that 
$C=(qu,xzv)$ is a complete intersection. The height condition and the expression $q=l_1x+l_2y$ yield that
\[ 
C=(x,y) \cap (x,u) \cap (x,l_2) \cap (z,q) \cap (z,u) \cap (v,q) \cap (u,v).
\]
Since 
$K=(x,y)\cap (u,v)\cap (z,q)$, it follows that
\[L' = C: K = (x,l_2) \cap (z,u) \cap (v,q) \cap (x,u).\]

Now, write $L'=L_1\cap L_2$, where $L_1=(x,l_2)\cap (z,u)\cap (v,q)=(z,u)\cap (q,xv,l_2v)=(z,u)(q,xv,l_2v)$ (by the modularity law and height conditions) and $L_2=(x,u)$.  
One checks that $\pd(R/L_1\oplus R/L_2)=\pd(R/L_1)=4$, and $\pd(R/(L_1+L_2))=\pd(R/(x,u,zl_2(v,y)))=4$; then 
Lemma \ref{ses} yields $\pd(R/L')\leq 4$, and hence $\pd(R/I)\leq 5$ by Theorem \ref{pd+1}(i).\\

\noindent \underline{Case 2: Assume ${\rm ht}(x,y,u,v,z)={\rm ht}(x,y,u,v)=4$.}\\ Then $z\in (x,y)+(u,v)$, so we can write $z=x' + u'$ with $x'\in (x,y)$ and $u'\in (u,v)$.  Since $\h(x,y,z) = \h(u,v,z) = 3$, both $x'$ and $u'$ are nonzero linear forms.  Then, after a change of variables, we may assume $z=x+u$.

If ${\rm ht}(x,y,u,v,q)=5$, the statement follows by Proposition \ref{<=t}, because the generators of $K$ can be written in terms of a regular sequence of length $5$. We may then assume $q\in (x,y,u,v)$.
Since $z=x+u\in (z,q)$, we may actually assume $q\in(x,y,v)$, and we write $q=l_1x+l_2y+l_3v$ for linear forms $l_1,l_2,l_3$.

If ${\rm ht}(x,y,u,l_3)=3$,  then $l_3\in (x,y,u)$, and 
we can modify $l_1$ and $l_2$ so that $l_3\in (u)$; using $z=x+u$ and modifying again $l_1$ to clear the term in $u$ we may take $l_3=0$, i.e. $q=l_1x+l_2y$. 
Then 
$$K=(x,y)\cap (u,v)\cap (x+u,l_1x+l_2y).$$

 If ${\rm ht}(x,y,u,v,l_1,l_2)\leq 5$, the statement follows by Proposition \ref{<=t}. If  instead ${\rm ht}(x,y,u,v,l_1,l_2)=6$, one can check that $\pd(R/K)=\pd(R/L')=3$ for a link $L'$ of $K$, thus $\pd(R/I)\leq 4$ by Theorem \ref{pd+1}(i). A similar argument proves the statement when ${\rm ht}(x,u,v,l_2)=3$ or $\h(y,v, l_1)=2$.

We may then assume ${\rm ht}(x,y,u,l_3)={\rm ht}(x,u,v,l_2) =4$ and $\h(y,v,l_1)=3$. Recall that $K_1=(x,y)\cap (u,v)$.  By the height conditions we have
\[ K_1+(x+u)= (x,u,y) \cap (x,u,v) \cap (x+u, u^2, v,y).\]
Since $(x+u, u^2, v,y)$ is $(x,y,u,v)$-primary of multiplicity two, then either $(x+u, u^2, v,y):q$ is $(x,u,v,y)$ or it is the unit ideal. In either case, 
$J=\left[K_1+(x+u)\right]:q=\left[(x,u,y) \cap (x,u,v)\right] :q$ and then the height assumptions yield $J=(x,u,v)\cap (x,u,y)=(x,x+u,yv)$.
By Remark \ref{observ}(b) we have $K=K_1 \cap \big[(x+u)+qJ \big]$;  since $q=l_1x+l_2y+l_3v$, setting $c=-l_1ux-l_2uy + l_3xv\in K_1$, we have the equality $(qx, x+u)=(c,x+u)$ and then by the modularity law 
\begin{align*}
K& =  K_1\cap \big[(x+u)+q(x,x+u,yv)\big] \\
& = (qyv)+ \left( K_1 \cap (qx,x+u)\right) \\
& = (qyv)+ \left( K_1 \cap (c,x+u)\right)  \\
&=   (qyv, c)+ (x+u)K_1.
\end{align*}
If $K=I^{un}$, observe that $[K]_3=(c)+(x+u)K_1$, then after possibly taking linear combinations we may assume two minimal generators of $I$ are multiples of $x+u$. The statement now follows by Corollary \ref{x}.

If $K = L$, then we consider the short exact sequence
$$0 \lra R/K \lra R/K_1\oplus R/(x+u, q)\lra R/(x+u, q, xu,xv,yu,yv)\lra 0;$$
observe that $(x+u, q, xu,xv,yu,yv)=(x+u, q, x^2, xy, xv,yv)$ and $\pd(R/K_1\oplus R/(x+u,q))=3$, thus $\pd(R/K)\leq \max\{3, \pd(R/(x+u, q, x^2, xy, xv,yv)) -1 \}=\max\{3, \pd(R/H)\}$,  where $H=(q, x^2, xy, xv, yv)$. Now, $\pd(R/(H,x))=\pd(R/(x, q, yv))\leq 3$ and $H:x=(x,y,v)+[(q,yv):x]$. 
Since $\h(y,v,l_1)=3$, it follows that $\h(y,v,q)=3$; thus $(q,yv)=(q,y)\cap (q,v)$. 
Moreover, since $\h(x,y,u,l_3)=4$ and $\h(x,y,u,v)=4$, it follows that $\h(x,y,q)=\h(x,y,l_3v)=3$ and so
\[(q,yv):x\subseteq (q,y):x =(q,y)\subseteq (x,y,v)\]
and therefore
$H:x=(x,y,v)$. Then $\pd(R/(H:x))=3$ and by Lemma \ref{ses}(1), $\pd(R/H)\leq 3$.  So by Lemma \ref{ses}(2), $\pd(R/K)\leq 4$, and Theorem \ref{pd+1}(i) yields $\pd(R/I)\leq 5$. 
\end{proof}

\subsection{Type $\boldsymbol{\langle 1,1,1;1,1,2\rangle}$}

The goal of this subsection is to prove the following result.
\begin{prop}\label{111;112}
Let $\p_1,\p_2,\p_3$ be distinct linear primes of height two. Assume $K=K_1\cap \p_2\cap \p_3$, where $K_1$ is a $\p_1$-primary ideal of multiplicity two.
If either $K=L$ or $K=I^{un}$, then $\pd(R/I)\leq 5$.
\end{prop}

We begin by noticing that ${\rm ht}(\p_1+ \p_2 + \p_3)\leq 5$.
\begin{lem}\label{ht6}
Let $\p_1,\p_2,\p_3$ be distinct linear primes of height two. Assume $K=K_1\cap \p_2\cap \p_3$, where $K_1$ is a $\p_1$-primary multiplicity two ideal.
If either $K=I^{un}$ or $K=L$, then ${\rm ht}(\p_1 + \p_2 + \p_3)\leq 5$.
\end{lem}
\begin{proof} 
Assume toward a contradiction that $\h(\p_1+\p_2+\p_3)=6$, and write $\p_1=(u,v)$. By \cite[Proposition~11]{En1}, either $K_1=(u,v^2)$ or $K_1=(u,v)^2+(au+bv)$ for forms $a,b$. If $K_1=(u,v^2)$, then the height condition implies $K=(u,v^2)\p_2\p_3$, so $[K]_3\subseteq (u)$, contradicting Remark \ref{observ}(a). 

We may then assume $K_1=(u,v)^2+(au+bv)$. Let $K_2=\p_2\cap \p_3=\p_2\p_3$. 
By Remark \ref{observ}(b), $K=K' \cap K_2$, where $K'=(u,v)^2 + (au+bv)J $ and $J=\left((u,v)^2 + K_2\right):(au+bv)$.  The height condition implies that a primary decomposition of $(u,v)^2 + K_2$ is $((u,v)^2 + \p_2) \cap ((u,v)^2 + \p_3)$.  Since $u, v\in J$, either $J=R$ or $\p_1+\p_2$ or $\p_1+\p_3$ or $\p_1+K_2$. In any case, since $(au+bv)J \subseteq (u,v)^2 + K_2$ and $(u,v)^2\subseteq K'$, we can write $K'=(u,v)^2 + (au+bv)J'$, where $(au+bv)J' +(u,v)^2 = (au+bv)J + (u,v)^2$ and $(au+bv)J' \subseteq K_2$. Then by the modularity law and the height conditions
\[
K=K' \cap K_2 = (au+bv)J' + \left((u,v)^2 \cap K_2\right)= (au+bv)J' + (u,v)^2K_2
\]
Therefore, $[K]_3\subseteq (au+bv)$, which contradicts Remark \ref{observ}(a).
\end{proof}


The following lemma will be useful to simplify the proofs of the following two results.
\begin{lem}\label{help}
Proposition~\ref{111;112} holds if $\h(u,v,z,w)=\h(u,v,x,y)=3$, or if all cubics in $K$ are contained in $H=(u,v)^2 \cap (x,y) \cap (z,w)$.
\end{lem}

\begin{proof} First, suppose ${\rm ht}(u,v,z,w)={\rm ht}(u,v,x,y)=3$.  After possibly a change of coordinate, we may assume $w\in (u,v)$ and $y\in (u,v)$. Then the quadric $yw\in K$ and the statement follows by Theorem~\ref{pd+1}(iii).

Next, assume all cubics in $K$ lie in $H$. We show this implies $\h(u,v,x,y)=\h(u,v,z,w)=3$. Indeed, assume to the contrary that $\h(u,v,x,y)=4$; then $H=(u,v)^2(x,y) \cap (z,w)$. If $\h(u,v,x,y,z,w)=6$, then $H=(u,v)^2(x,y)(z,w)$, which contains no cubics, a contradiction. Now assume $\h(u,v,x,y,z,w)=5$; then we may assume $z \notin (u,v,x,y)$.  It follows that all cubics in $H$ are multiples of $w$, contradicting Remark \ref{observ}(a).


We may then assume $\h(u,v,x,y,z,w)=4$, and so $(z,w)\subseteq (u,v,x,y)$. By Proposition~\ref{inters}, we may assume $\h(x,y,z,w)=4$.   Since $\h(u,v,z,w)\geq 3$, after a change of coordinates we may either assume $w=u$ and $z=y+v$, or $w=u+x$ and $z=y+v$. In the latter case
\[ H = (u,v)(vx-uy) + y(u,v)^2(y+v), u^2(u+x)(x,y)\]
while in the former case  $H = (u^2, uv) (x,y) + v^2(x,y)(y+v)$. 
In either case $[H]_3$ is contained in a principal ideal, contradicting Remark \ref{observ}(a).
\end{proof}

By \cite[Proposition~11]{En1}, either $K_1=(u,v^2)$ or $K_1=(u,v)^2+(au+bv)$ for forms $a,b$. The former case will be worked out at the very end of this subsection, while the latter one requires more work and will be addressed in the next two results.

\begin{lem}\label{111;112deg1}
Proposition~\ref{111;112} holds if $K_1=(u,v)^2+(au+bv)$ and $\deg(a)=\deg(b)=1$.
\end{lem}
\begin{proof}
Let $\p_1=(u,v)$, $\p_2=(x,y)$, $\p_3=(z,w)$, and $K_1=(u,v)^2+(au+bv)$ for linear forms $a,b$ with $\h(x,y,a,b)=4$.
Since $\p_2=(x,y)$ and $\p_3=(z,w)$ are distinct primes, we have that ${\rm ht}(x,y,z,w)\geq 3$. If ${\rm ht}(x,y,z,w)=3$, we may assume $(x,y)\cap (z,w)=(x,yz)$, hence $K=(x,yz)\cap K_1$, and the statement follows by Proposition~\ref{inters}.
Thus, we may assume ${\rm ht}(x,y,z,w)=4$. \\

\noindent \underline{Case 1: Assume ${\rm ht}(u,v,a,b,z,w)=6$.}\\
 Then $K=(x,y)\cap K_1(z,w)=(x,y)(z,w)\cap K_1(z,w)$. 

 If $(au+bv)(z,w)\cap \big[(x,y)(z,w)+(u,v)^2(z,w)\big]\subseteq \m(au+bv)(z,w)$ where $\m$ is the homogeneous maximal ideal of $R$, then 
\[[K]_3\subseteq\big[ (x,y)(z,w)+(u,v)^2(z,w)\big]_3 = \big[ (u,v)^2 \cap (x,y) \cap (z,w)\big]_3.\]
and the statement follows by Lemma \ref{help}.

 Then, after possibly replacing $z$ by a linear combination of $z$ and $w$ we may assume $(au+bv)z\in (x,y)(z,w)+(u,v)^2(z,w)$. After possibly modifying $a$ and $b$ modulo $(u,v)z$ we may further assume $(au+bv)z\in (x,y)(z,w)$, hence $au+bv\in (x,y)$. Observe that ${\rm ht}(x,y,u,v)\leq 3$. Indeed, if ${\rm ht}(x,y,u,v)=4$, then $(a,b)\subseteq (x,y)$, and since ${\rm ht}(u,v,a,b)=4$ and $a,b$ are linear forms, then $(a,b)=(x,y)$. Thus ${\rm ht}(u,v,z,w,x,y)=6$, which is ruled out by Lemma \ref{ht6}. Therefore, $au+bv\in (x,y)$ and ${\rm ht}(x,y,u,v)=3$.
Without loss of generality, from the height condition we may assume $v=x$, then $au\in (x,y)$, yielding $a\in (x,y)$. Since ${\rm ht}(u,v,a,b)=4$ and $v=x$, we have $a\notin (x)$, thus $(x,y)=(x,a)$, so after a change of coordinate we may further assume $y=a$. Then
\begin{align*}
K&=(x,y)\cap (z,w)\big((u,x)^2+(yu+bx)\big)\\
&=(xwu, xzu, x^2w, bxw + ywu, x^2z, bxz + yzu),
\end{align*}
where the last equality follows because ${\rm ht}(u,x,y,z,w,b)=6$. One checks that $\pd(R/K)=3$.  Moreover, if we set $L':=(x^2w,bxz+yzu):K=(xzw, yzw, x^2w, bxz + yzu, bx^3 + yx^2u)$, then $\pd(R/L') = 3$ also.  By Theorem \ref{pd+1}(i), we obtain $\pd(R/I)\leq4$.

The case ${\rm ht}(u,v,a,b,x,y)=6$ is argued similarly. Therefore, we may assume ${\rm ht}(u,v,a,b,z,w)\leq 5$ and ${\rm ht}(u,v,a,b,x,y)\leq 5$. Moreover, if any of these two heights is four, then ${\rm ht}(x,y,z,w,u,v,a,b)\leq 5$, and the statement follows by Proposition \ref{<=t}. \\

\noindent \underline{Case 2: Assume ${\rm ht}(u,v,a,b,z,w)={\rm ht}(u,v,a,b,x,y)=5$.}\\
 If $\h(u,v,z,w)=4$, then we may assume $b\in (u,v,z,w,a)$ and $\h(u,v,z,w,a)=5$. After possibly a change of the $u,v$ variables and clearing terms in $u,v$ using that $(u,v)^2\subseteq K_1$, we may further assume $b\in (z,w)$ and so we may take $b=w$. 
By Remark \ref{observ}(b) we have $K_1\cap (z,w)=K_1\cap \big[(z)+wJ\big]$, where $J=\big[K_1+(z)\big]:w$. Since $\h(u,v,z,w,a)=5$, the ideal $K_1+(z)$ is $(u,v,z)$-primary, and since $w\notin (u,v,z)$, it follows that $J=K_1+(z)$. So
$$K_1\cap (z,w)=K_1\cap \big[(z)+wK_1\big]=wK_1+\big[K_1\cap(z)\big]=wK_1 + zK_1$$
Then $[K]_3\subseteq [(w,z)K_1]_3$; since all cubics in $(w,z)K_1$ are written in terms of $u,v,a,b,z,w$ and ${\rm ht}(u,v,a,b,z,w)=5$, if $K = I^{un}$, then $\pd(R/I)\leq 5 $ by Remark \ref{observ}(c). If $K=L\sim I^{un}$, let $I'=(wu^2,zv^2):[K]_3=(wu^2, zv^2): (z,w)K_1$. Since $\h(u,v,z,w,a)=5$, one checks that $I'$ contains no quadrics and only two linearly independent cubics, and so does $I^{un}$, by Theorem \ref{linkage}(d). This is a contradiction. One argues similarly if ${\rm ht}(u,v,x,y)=4$. 

Then we may assume ${\rm ht}(u,v,z,w)={\rm ht}(u,v,x,y)=3$, and the statement follows by Lemma \ref{help}.
\end{proof}

\begin{lem}\label{111;112deg2}
Proposition~\ref{111;112} holds if $K_1=(u,v)^2+(au+bv)$ and $\deg(a)=\deg(b)\geq 2$.
\end{lem}
\begin{proof}
If $K$ contains a quadric the statement follows by Theorem~\ref{pd+1}(iii). Hence we may assume $K$ is generated in degree three and higher. Also, ${\rm ht}(u,v,x,y,z,w)\leq 5$ by Lemma \ref{ht6}. As in the proof of Lemma~\ref{111;112deg1}, by Proposition~\ref{inters} we may assume ${\rm ht}(x,y,z,w)=4$, then, by Lemma \ref{help}, we may also assume $\h(u,v,x,y)=4$ and 
$[K]_3 \neq \big[  (u,v)^2\cap (x,y) \cap (z,w) \big]_3$. In particular, we have $\deg(a) = \deg(b) = 2$ and $au+bv\in \big((x,y)\cap (z,w)\big)+(u,v)^2$, and hence by Lemma \ref{cubic}(ii) we may assume $au+bv\in (x,y)(z,w)$ and, by the modularity law, 
\[K=(au+bv)+\big[(u,v)^2\cap (x,y)\cap (z,w)\big].\]
Moreover, observe that $au\in (x,y,v)$ and since $\h(x,y,u,v)=4$, then $a\in (x,y,v)$. After clearing the term in $v$, we may further assume $a\in (x,y)$. By symmetry, $b\in (x,y)$. 
Now, if $\h(u,v,z,w)=4$, the same argument yields $a,b\subseteq (z,w)$, then $a,b\subseteq (x,y)\cap (z,w)=(x,y)(z,w)$, and since 
$a,b$ are quadrics, then all generators of $K$ lie in $k[u,v,x,y, z,w]$ and since $\h(u,v,x,y,z,w)\leq 5$, then $\pd(R/I)\leq 5$ by Proposition \ref{<=t}. 

We may then assume ${\rm ht}(u,v,z,w)=3$, and after a change of coordinates we may take $w=v$ and $\h(u,v,z)=3$. Note that $\h(x,y,z,v)=4$, because $v=w$. Since $au+bv\in (z,v)$, it follows that $au\in (z,v)$ and so $a\in (z,v)$. Then $a\in (z,v)\cap (x,y)=(z,v)(x,y)$. 
Then 
\begin{align*}
K&=(au+bv)+\big[(u,v)^2\cap (x,y)\cap (z,v)\big]\\
&=(au+bv) + v(u,v)(x,y)+\big[u^2(x,y)\cap (z,v)\big].
\end{align*} 
Since $\h(u,v,x,y)=4$, we have $u^2(x,y)\cap (z,v)=(u^2) \cap (x,y)\cap (z,v)=u^2 J$, where $J=\left((x,y)\cap(z,v)\right):u^2$. Since $\h(u,x,y)=\h(u,z,v)=3$ and $\h(x,y,z,v)=4$, we have $J=(x,y)\cap (z,v)=(x,y)(z,v)$, therefore
\[ K = (au+bv)+v(u,v)(x,y)+u^2(x,y)(z,v).\]
%
Since $[K]_3=(au+bv) + v(u,v)(x,y)$, if $K=I^{un}$, then after possibly taking linear combinations, we may assume two of the generators of $I$ are multiple of $v$, and so $\pd(R/I)\leq 4$ by Corollary \ref{x}. 
Assume then $K=L\sim I^{un}$. 
The primary decomposition of $K$ and the fact that $v=w$ yields that $K:v=(u,v)\cap (x,y)$, whence $\pd(R/(K:v))=3$. On the other hand,
$K+(v)=(v, u(a, uzx, uzy))$ and one sees $\pd(R/K+(v))\leq 4$. 
Then, by Lemma \ref{ses}, $\pd(R/K)\leq 4$, and then $\pd(R/I)\leq 5$ by Theorem \ref{pd+1}(i).
\end{proof}

{\bf Proof of Proposition~\ref{111;112}.} If $K_1$ does not contain a linear form we invoke Lemmas~\ref{111;112deg1} and \ref{111;112deg2}. If $K_1$ contains a linear form we may assume $K=(u,v^2)\cap (x,y)\cap (z,w)$. Note that ${\rm ht}(x,y,u,v,z,w)\leq 5$ by Lemma \ref{ht6}, then Proposition \ref{<=t} yields $\pd(R/I)\leq 5$. 
\QED
\bigskip

\subsection{Type $\boldsymbol{\langle 1,1,1,1;1,1,1,1\rangle}$}

In this subsection we cover the case when $I^{un}$ or $L = (f_1,f_2):I$ is the intersection of four distinct linear primes.

\begin{lem}\label{p_i}
Let $\p_1,\p_2,\p_3$ be distinct linear primes of height $2$, and assume ${\rm ht}(\p_1+\p_2+\p_3)\geq 5$. If ${\rm ht}(\p_i+\p_j)=4$ for all $i\neq j$ then $\p_1\cap \p_2\cap \p_3=\p_1\p_2\p_3$.
\end{lem}

\begin{proof}
The statement is clear if ${\rm ht}(\p_1+\p_2+\p_3)=6$, hence we may assume ${\rm ht}(\p_1+\p_2+\p_3)=5$. Let $\p_1=(x,y)$, $\p_2=(z,w)$, and $\p_3=(u,v)$. We may assume $\h(x,y,z,w,u)=5$ and $v\in (x,y,z,w,u)$. Since $\p_3=(u,v)$, we may assume further assume $v\in (x,y,z,w)$. Since $\h(\p_1+\p_3)=\h(\p_2+\p_3)=4$, we have $v=x'+z'$ for some $0\neq x'\in (x,y)$ and $0\neq z'\in (z,w)$, thus after possibly choosing different minimal generators for $\p_1$ and $\p_2$ we may assume $v=x+z$. Since $\h(x,y,z,w,u)=5$, if we set $H:=\p_1\p_2 + (u)=\big(\p_1+(u)\big)\cap \big(\p_2+(u)\big)$, then $x+z$ is regular on $R/\big(\p_i+(u)\big)$ for $i=1,2$. Thus $H:(x+z)=H$ and hence 
\begin{align*}
\p_1\cap \p_2\cap \p_3 &=\p_1\p_2\cap (u,x+z) & \text{ since } \h(\p_1+\p_2) = 4 \\
&=\p_1\p_2\cap \left[(u) + (x+z)H\right]& \text{ by Remark~\ref{observ}(b)}\\
&=\p_1\p_2\cap \left[(u) + (x+z)\p_1\p_2\right] \\
&=\left[\p_1\p_2 \cap (u) \right] + (x+z)\p_1\p_2&\text{ by the modularity law}\\
&=(u)\p_1\p_2  + (x+z)\p_1\p_2 & \text{ since } \h(\p_1 + \p_2 + (u)) = 5\\
&=\p_1\p_2\p_3.
\end{align*}
\end{proof}

\begin{prop}\label{1111;1111}
Let $K=\p_1\cap \p_2\cap \p_3\cap \p_4$, where the $\p_i$ are distinct linear primes of height $2$. If either $K=I^{un}$ or $K=L$, then $\pd(R/I)\leq 5$.
\end{prop}

\begin{proof}
By Theorem~\ref{pd+1}(iii) we may assume $K$ contains no quadrics. Also, if ${\rm ht}(\p_1+\p_2+\p_3+\p_4)\leq 5$, the statement follows by Proposition \ref{<=t}. We may then assume ${\rm ht}(\p_1+\p_2+\p_3+\p_4)\geq 6$.
Since the primes are distinct, we have ${\rm ht}(\p_i+\p_j)\geq 3$  for every $i\neq j$. If ${\rm ht}(\p_i+\p_j)=3$ for some $i\neq j$, then we are in the assumptions of Proposition~\ref{112;111}. We may then assume ${\rm ht}(\p_i+\p_j)=4$ for every $i\neq j$.
In the rest of the proof we show it is impossible that $\h(\p_1+\p_2+\p_3+\p_4)\geq 6$ and ${\rm ht}(\p_i+\p_j)=4$ for every $i\neq j$.

Suppose $\h(\p_i+\p_j+\p_k) = 4$ for some distinct $i,j,k$.  Then given the height restrictions, after a linear change of variables $K$ must have the form
\[K = (x,y) \cap (w,z) \cap (x+w,y+z) \cap (a,b)\]
for distinct linear forms $w,y,w,z,a,b$.  One checks that all cubics in $K$ are multiples of $wy - xz$, which is impossible.  

If $\h(\p_i + \p_j + \p_k) = 6$ for all $i,j,k$, then after a linear change of variables $K$ must have one of the forms:
\begin{align*}
 & (x,y) \cap (w,z) \cap (a,b) \cap (x+w+a,y+z+b) & \text {if } \h(\p_1 + \p_2 + \p_3 + \p_4) = 6,\\
 & (x,y) \cap (w,z) \cap (a,b) \cap (x+w+a,u)& \text {if } \h(\p_1 + \p_2 + \p_3 + \p_4) = 7,\\
& (x,y) \cap (w,z) \cap (a,b) \cap (u,v)& \text {if } \h(\p_1 + \p_2 + \p_3 + \p_4) = 8,
 \end{align*}
 where all variables are distinct linear forms.  In all three cases $K$ contains no cubics, which is impossible.

Therefore we may assume, in addition to the previous height restrictions, that $\h(\p_1 + \p_2 + \p_3) = 5$, $\p_4 = (a,b)$ and $a \notin \p_1 + \p_2 + \p_3$.  Then $K_1 :=\p_1\cap \p_2\cap \p_3 =\p_1\p_2\p_3$ by Lemma \ref{p_i}.  Since $a \notin \p_1 + \p_2 + \p_3$, it follows that $K_1 + (a) = \p_1\p_2\p_3 + (a) = \bigcap_{i = 1}^3 \big(\p_i + (a)\big)$.  
Then 
 
 \begin{align*}
 K &= \bigcap_{i = 1}^4 \p_i\\
 &= \p_1\p_2\p_3 \cap (a,b) &  \text{since } \bigcap_{i = 1}^3 \p_i = \p_1\p_2\p_3\\
 &= \p_1\p_2\p_3 \cap \big[(a) + b\big(\p_1\p_2\p_3 + (a)\big):b\big] &\text{ by Remark~\ref{observ}(b)}\\
 &= \p_1\p_2\p_3 \cap \left[(a) + b\left( \bigcap_{i = 1}^3 (\p_i + (a)):b\right)\right]& \text{ by above}\\
 &=  \p_1\p_2\p_3 \cap \left[(a) + b\left( \bigcap_{i = 1}^3 (\p_i + (a))\right)\right]& \text{ since } \h(\p_i + (a,b)) = 4\\
  &=  \p_1\p_2\p_3 \cap \big[(a) + b\left( \p_1\p_2\p_3 + (a)\right)\big] & \text{ by above again}\\
    &=  \p_1\p_2\p_3 \cap \big[(a) + b \p_1\p_2\p_3 \big]\\
    &= a\p_1\p_2\p_3 + b\p_1\p_2\p_3\\
        &=\p_1\p_2\p_3\p_4,
\end{align*}
which is generated in degree $4$ and again impossible.  This completes the proof.
\end{proof}

\bigskip

\subsection{Type $\boldsymbol{\langle 1,1;1,3\rangle}$}
We prove the case where $I^{un}$ or the link $L$ is of type $\langle 1,1;1,3\rangle$ in the next two results.

\begin{prop}\label{L11;13}
If $L$ 
is of type $\langle 1,1;1,3\rangle$, then $\pd(R/I) \le 5$.
\end{prop}

\begin{proof} We have $L = J \cap (u,v)$, where $J$ is $(x,y)$-primary for independent linear forms $x$ and $y$ and has one of the eight forms from \cite[Theorem~2.1]{MM}, and $u,v$ are independent linear forms.\\
\\
\noindent\underline{Cases (i), (ii), (iii): $J = (x,y)^2, (x,y^3)$, or $(x^2, xy, ax+y^2)$.}\\
In all three cases, all minimal generators of $J \cap (u,v)$ can be written in terms of at most $5$ linear forms. By Lemma \ref{<=t}, $\pd(R/I) \le 5$.\\

\noindent\underline{Case (iv): $J = (x^2, xy, y^3, ax + by^2)$, where $\h(x,y,a,b) = 4$.}\\
If $\h(x,y,u,v)<4$, then $L$ contains a quadric and so $\pd(R/I)\leq 4$ by Theorem \ref{pd+1}(iii). We may then assume $\h(x,y,u,v)=4$, and in particular $(x^2,xy,y^3)\cap (u,v)=(x^2,xy,y^3)(u,v)$.
Set $f=ax+by^2$. If $\deg(f)=2$, then $\deg(a) =1$ and $b \in K$, so 
$L=L'R$ where $L'$ is an ideal in $R'=k[x,y,a,u,v]$. By Proposition \ref{<=t} we obtain $\pd(R/I)\leq 5$. 
We may then assume $\deg(f)\geq 3$.
If $f\notin (x^2, xy, y^3)+(u,v)$, setting $J=(x^2,xy,y^3,u,v):f$ one obtains that $fJ$ is generated in degree at least four; then by Remark \ref{observ}(b) and the above all cubics of $L$ are contained in $(x^2,xy,y^3)(u,v)$ and so $[L]_3\subseteq (x)$, which contradicts Remark \ref{observ}(a).
We may then assume $f\in (x^2, xy, y^3)+(u,v)$, and a proof similar to Lemma \ref{cubic}(ii) yields that we may assume $f\in (u,v)$. Then $J + (u,v)=(x^2,xy,y^3,u,v)$, thus $\pd(R/J+(u,v))=4$. Since $\pd(R/J)\leq 4$ by \cite[Theorem~2.1]{MM} and $\pd(R/(u,v)) = 2$, Lemma \ref{ses}(2) yields $\pd(R/L)\leq 4$ and then $\pd(R/I)\leq 5$ by Theorem \ref{pd+1}(i).\\

\noindent\underline{Case (v): $J = (x,y)^3 + (ax + by)$, where $\h(x,y,a,b) = 4$.}\\
If $\h(x,y,u,v) = 3$ we may further assume  $u = x$.  In this case we observe that $\pd(R/J+(u,v))=\pd(R/(x, y^3, by, v))\leq 4$ and since $\pd(R/J)\leq 3$, 
  by Lemma \ref{ses}(2) we obtain $\pd(R/L) \le 3$ and hence $\pd(R/I) \le 4$ by Theorem \ref{pd+1}(i). 
  
  If $\h(x,y,u,v) = 4$, then $(x,y)^3\cap (u,v)=(x,y)^3(u,v)$. When $\deg(a)=\deg(b)\geq 2$, by Lemma \ref{cubic}(i) all cubics in $L$ lie in a principal ideal, contradicting Remark \ref{observ}(a). When $\deg(a)=\deg(b)=1$, if $\h(x,y,u,v,a,b)\leq 5$ the statement follows by Proposition \ref{<=t}, while if $\h(x,y,u,v,a,b) = 6$, then $[L]_3\subseteq (ax+by)$, contradicting Remark \ref{observ}(a). \\


\noindent\underline{Case (vi): $J = \left(x,y\right)^3 + (x,y)(ax+by) + \left(c(ax+by)+dx^2+exy+fy^2\right)$,}\\
\underline{where $\h(x,y,a,b)= 4$ and  $\h(x,y,c,b^2d-abe+a^2f)=4$.}\\
Notice that $a,b,c,d,e,f$ are linear forms, since otherwise $[L]_3=[(x,y)^3\cap (u,v)]_3$ and so either $L$ contains no cubics (if $\h(x,y,u,v)=4$) or all cubics in $L$ are contained in a principal ideal (if $\h(x,y,u,v)=3$), contradicting in either case Remark \ref{observ}(a). 

Assume $\h(x,y,u,v) = 3$, in which case we may take $u = x$.  Observe that $J + (u,v) = (x, v, y(y^2, by, bc+fy))$ can be written in terms of at most $6$ variables.  If $\h(x,y,v,b,c,f) = 6$, then one checks that $\pd(R/(J + (u,v))) = 4$ (e.g. by iterated applications of Lemma \ref{ses}(1)).  If $\h(x,y,v,b,c,f) \le 5$, then $J+(u,v)$ is extended from a polynomial ring in at most 5 variables, thus $\pd(R/(J + (u,v))) \le 5$. In either case, since $\pd(R/(u,v))=2$ and $\pd(R/J)\leq 4$, by Lemma \ref{ses}(2) we have $\pd(R/L) \le 4$, so Theorem \ref{pd+1}(i) yields $\pd(R/I) \le 5$.

We may then assume $\h(x,y,u,v)=4$. We claim $\h(x,y,u,v,a,b)\leq 4$. To the contrary assume $\h(x,y,u,v,a)=5$; we show that $H=(u,v)\cap \left[(x,y)^3+ (x,y)(ax+by)\right]$ contains no cubics. If $0\neq s\in H$ is a cubic, then $s=F+ x' (ax+by)$ for some $F\in (x,y)^3$ and some linear form $0\neq x'\in (x,y)$. After possibly a change of variables, we may assume $x'=x$. Since $H\subseteq (u,v)\cap (x,y)^2=(u,v)(x,y)^2$, we have $s\in (u,v)(x,y)^2$ and so  
 $x(ax+by)\in (x,y)^2(x,y,u,v)$. Then $ax^2\in (xy,y^2,x^3,x^2u,x^2v)$, yielding $a\in (x,y,u,v)$, which is a contradiction. Thus $H$ contains no cubics, so by Lemma \ref{cubic}(i) the ideal $L$ contains at most one cubic, contradicting Remark \ref{observ}. 

We may then assume $\h(x,y,u,v,a,b)=4$, so $(a,b)\subseteq (x,y,u,v)$, and since $(x,y)^2\subseteq J$, we may assume $(a,b)\subseteq (u,v)$. Since $\h(x,y,a,b)=4$, without loss of generality we may assume $a=u$ and $b=v$, so $L=(ax+by)(x,y)+[(a,b)\cap ((x,y)^3+(s))]$, where $s=c(ax+by)+dx^2+exy+fy^2$.  
If $s\notin (a,b)+(x,y)^3$, then Remark \ref{observ}(b) implies that all cubics of $L$ are contained in $(ax+by)(x,y)+[(a,b)\cap (x,y)^3]$. Since $[(a,b)\cap (x,y)^3]=(a,b)(x,y)^3$ is generated in degree 4, the cubics of $L$ are all contained in $(ax+by)(x,y)$, contradicting Remark \ref{observ}(a). Therefore, $s\in (a,b) + (x,y)^3$ and hence $s':=dx^2+exy+fy^2\in (a,b)+(x,y)^3$. Since $(x,y)^3\subseteq J$, after possibly clearing from $d,e,f$ all terms in $x$ and $y$, we may assume $s'\in (a,b)$ and thus $s'\in (a,b)\cap (x,y)^2=(a,b)(x,y)^2$. Therefore we may assume $s'\in k[x,y,a,b]$, and then a minimal generating set of $L$ lies in $k[x,y,a,b,c]$. Then by Proposition~\ref{<=t} we have $\pd(R/I)\leq 5$.\\

\noindent\underline{Case (vii): $J = \left(x,y\right)^3  +(x,y)(ax + by) + \left(a(ax+by)+cxy+dy^2\right)$}\\
\noindent\underline{$+ \left( b(ax+by)-cx^2-dxy \right)$, with $\h(x,y,a,b) = 4$ and $\h(x,y,a,b,c,d) \ge 5$.}\\
As in the previous case, we may assume $a,b,c,d$ all have degree $1$. Moreover, as in Case (vi), if $\h(x,y,u,v)=3$, we may assume $u=x$, so $J+(u,v)=(x,v,y(y^2, by, ab+dy, b^2))$. 
If $\h(x,y,v,a,b,d)\leq 5$, then the statement follows by Proposition \ref{<=t}; if $\h(x,y,v,a,b,d)$ one checks that $\pd(R/J+(u,v))=5$, and since $\pd(R/J)\leq 4$, then by Lemma \ref{ses}(2) we have $\pd(R/L)\leq 4$, so $\pd(R/I)\leq 5$ by Theorem \ref{pd+1}.

Assume $\h(x,y,u,v)=4$ and let $s_1=a(ax+by)+cxy+dy^2$ and $s_2=b(ax+by)-cx^2-dxy$. As above, we first prove $\h(x,y,u,v,a,b)\leq 4$. If not, we may assume that $\h(x,y,u,v,a)=5$.  As above, this implies $H=(u,v)\cap [(x,y)^3 + (x,y)(ax+by)]$  contains no cubics. Then applying Remark \ref{observ}(b) first to $s_2$ and then to $s_1$ we see that either $L$ does not contain two linearly independent cubics or else both $s_1$ and $s_2$ lie in $(u,v) + (x,y)^3 + (x,y)(ax+by)$. Since $(x,y)^3 + (x,y)(ax+by)\subseteq J$, we may modify $c$ and $d$ modulo $(x,y)$ to assume $s_1\in (u,v) + (x,y)(ax+by)$. This implies $a^2\in (y,u,v,ax^2)$, thus $a\in (y,u,v,x)$, contradicting the height assumption. 

Therefore, $\h(x,y,u,v,a,b)=4$ and as above we may assume $a=u$ and $b=v$.  If $\h(x,y,a,b,c,d) = 6$, then all cubics in $L$ are multiples of $ax+by$, a contradiction.  Therefore $J$ is extended from a polynomial ring in $5$ variables, so $\pd(R/I) \le 5$ holds by Proposition~\ref{<=t}.\\



\noindent\underline{Case (viii): $J  = \left(x,y\right)^3 +  (x,y)(ax+by) + [J]_{\ge 4}$,  with $\h(x,y,a,b) = 4$.}\\
Observe that all cubics are contained in $H:= [(x,y)^3+(x,y)(ax+by)]\cap (u,v)$. If $\h(x,y,u,v)=4$
 it follows that all cubics in $L$ are multiples of $ax+by$, contradicting Remark \ref{observ}(a).  
 

Then $\h(x,y,u,v)=3$ and we may assume $u=x$; then 
\[L=(x^3,x^2y,xy^2,x(ax+by))+\left[(x,v)\cap \left(y(y^2,ax+by), J_{\geq 4}\right)\right].\]
 Since $\h(x,y,v)=3$, it follows that $(x,v)\cap (y(y^2,ax+by))=(y)[(x,v)\cap (y^2,ax+by)]$. Now, if $\h(x,y,v,b)=4$, then $ax+by\notin (x,v,y^2)$, thus by Remark \ref{observ}(b), $H_1:=(x,v)\cap (y^2,ax+by)=J_1(ax+by)+(xy^2,vy^2)$ for some proper homogeneous ideal $J_1$; in particular $H_1$ is generated in degree 3 and higher.  Therefore $(x,v)\cap (y(y^2,ax+by))$ is generated in degree at least 4, and so $\left[(x,v)\cap \left(y(y^2,ax+by), J_{\geq 4}\right)\right]$ contains no cubics. Thus all cubics of $L$ are contained in $(x)$, contradicting Remark \ref{observ}(a).

We may then assume $\h(x,y,v,b)=3$, thus after a linear change of variables we may take $b=v$. Then  $L= L' + [L]_{\ge 4}$, where $[L]_{\ge 4}$ is generated in degree at least $4$ and $L' = (x^3,x^2y,xy^2,x(ax+by), y(ax+by), by^3)$. 
Recall $f_1,f_2$ are the regular sequence of two cubics in $I$ defining the link $L\sim I^{un}$, so $f_1,f_2\subseteq L'$ and in particular $f_1,f_2\in k[x,y,a,b]$. Set $J'=(x^3,y(ax+by)):L'=(x^2 , b^2y , bxy, axy+by^2)$, and $H'=(f_1,f_2):L'$, then $H'$ is extended from $k[x,y,a,b]$, and, by Theorem \ref{linkage}(d), $H'$ is generated by one quadric $q$ and forms of degree at least 3 lying in $k[x,y,a,b]$. Since $I^{un}\subseteq H'$, we have that $f_3=\ell q + F_3$ for some quadric $q$ and some cubic $F_3$, both lying in $k[x,y,a,b]$. It follows that $f_1,f_2,f_3\in k[x,y,a,b,\ell]$ and then $\pd(R/I)\leq 5$ by Proposition \ref{<=t}.
\end{proof}

\begin{prop}\label{11;13}
If $I^{un}$ is of type $\langle 1, 1; 1, 3 \rangle$, then $\pd(R/I) \le 5$.
\end{prop}

\begin{proof} 
 We have $I^{un} = J \cap (u,v)$, where $J$ is $(x,y)$-primary for independent linear forms $x$ and $y$ and has one of the eight forms from \cite[Theorem~2.1]{MM} and $u,v$ are independent linear forms.\\
Notice that in cases (vi)--(viii) below, we may assume $a,b,c,d,e,f\in R_1$, because otherwise all cubics in $I^{un}$ are contained in $(x,y)^3$, and thus $\pd(R/I)\leq 2$ by Remark \ref{observ}(c).\\

\noindent\underline{Cases (i), (ii), (iii): $J = (x,y)^2, (x,y^3)$, or $(x^2, xy, ax+y^2)$.}\\
In all three cases, all minimal generators of $J \cap (u,v)$ can be written in terms of at most $5$ linear forms. By Lemma \ref{<=t}, $\pd(R/I) \le 5$.\\

\noindent\underline{Case (iv): $J = (x^2, xy, y^3, ax + by^2)$, where $\h(x,y,a,b) = 4$.}\\
If $\h(x,y,u,v)\leq 3$, then $I^{un}$ contains a quadric and we are done by Proposition~\ref{pd+1}(iii). We may then assume $\h(x,y,u,v)=4$. 
Let $f=ax+by^2$. If $\deg(f)=2$, then $I^{un}$ is extended from $k[x,y,a,u,v]$, thus $\pd(R/I) \le 5$ by Proposition \ref{<=t}. If $\deg(f)\geq 4$, then all cubics in $I^{un}$ are extended from $k[x,y,u,v]$ and $\pd(R/I)\leq 4$ by Remark \ref{observ}(c). We may then assume $\deg(a)=2$ and $\deg(b)=1$. 
If $\h(x,y,u,v,a,b)\geq 5$ then $ax+by^2\notin (u,v,x^2,xy,y^3)$, thus by Remark \ref{observ}(b) all cubics of $I^{un}$ are contained in $(u,v)\cap (x^2,xy,y^3)$ and then are contained in $(x)$, a contradiction. So $a,b\subseteq (x,y,u,v)$. Since $(x^2,xy,y^3)\subseteq J$, we may assume $a,b\subseteq (u,v)$, and after a change of variables we may assume $b = u$. Then $I^{un} = (ax+uy^2)+(x^2, xy, y^3)(u,v)$, and thus every cubic in $I^{un}$ is contained in $(x,u)$, and so $I\subseteq (x,u)$. Then $(x,u)$ is a third minimal prime of $I$ contradicting the assumption on the type of $I^{un}$.\\

\noindent\underline{Case (v): $J = (x,y)^3 + (ax + by)$, where $\h(x,y,a,b) = 4$.}\\
If $\deg(a)=\deg(b)\ge2$, then $J$ is generated by elements of degree at least three, so all cubics of $I$ are extended from $k[x,y,a,b]$ which, by the height assumption, is a polynomial ring in 4 variables; thus $\pd(R/I)\leq 4$ by Remark \ref{observ}(c). We may then assume $\deg(a)=\deg(b)=1$. 
If $\h(x,y,u,v,a,b)=6$, then $I^{un}=(u,v)J$ so its cubics all lie in $(ax+by)$, which is a contradiction. Therefore $\h(x,y,a,b,u,v)\leq 5$, so $I^{un}$ is extended from a polynomial ring in at most 5 variables, thus $\pd(R/I)\leq 5$ by Proposition \ref{<=t}. \\

\noindent\underline{Case (vi): $J = \left(x,y\right)^3 + (x,y)(ax+by) + \left(c(ax+by)+dx^2+exy+fy^2\right)$,}\\
\underline{where $\h(x,y,a,b)= 4$ and  $\h(x,y,c,b^2d-abe+a^2f)=4$.}\\
First assume $\h(x,y,u,v)=4$.  As in Proposition \ref{L11;13}(vi), we must have $\h(x,y,u,v,a,b)\geq 4$. 
So $a,b\subseteq (x,y,u,v)$, and since $(x,y)^3\subseteq J$, we may assume $(a,b)\subseteq (u,v)$. Since $\h(x,y,u,v)=4$, after a change of coordinates we may assume $a=u$ and $b=v$. Then, setting $s=c(ux+vy)+dx^2+exy+fy^2$ we have
\[
I^{un}= (x,y)(ux+vy)+ \big(\left((x,y)^3 + (s)\right) \cap (u,v)\big).
\]
Since $\deg(s)=3$ and $(x,y)^3\cap (u,v)$ contains no cubics, if $s\notin (x,y)^3 + (u,v)$ then by Lemma \ref{cubic}(i) then $\left[I^{un}\right]_3\subseteq (ax+by)$, contradicting Remark \ref{observ}(a). Then $s\in (x,y)^3+(u,v)$, so there exists a cubic $s'\in J\cap (u,v)$ such that 
\[ I^{un}=(x,y)(ax+by) + \big( \left((x,y)^3+(s')\right)\cap (u,v) \big) = (x,y)(ax+by)+ (s') + (x,y)^3(u,v).\]
For degree reasons $I\subseteq (x,y)(ax+by)+ (s')$, so after possibly taking linear combinations of the generators of $I$, we may assume two of the three cubics in $I$ are multiple of $g=ax+by$, so the statement follows by Corollary \ref{x}.

We may then assume $\h(x,y,u,v) = 3$ and $x = u$, so 
\[
I^{un}=(x^3, x^2y,xy^2, x(ax+by)) + \left((y^3) + y(ax+by) + (F)\right) \cap (x,v).
\] 
where $F=c(ax+by)+dx^2+exy+fy^2$. 
If $\h(x,y,v,b)=4$, then no cubics lie in $\left(y^3, y(ax+by)\right)\cap (u,v)$ and so Lemma \ref{cubic}(i) yields $I\subseteq (x^3, x^2y,xy^2, x(ax+by), F')$ for some cubic $F'$. 
Then there is a minimal generating set of $I$ where two generators lie in $(x)$, and the statement follows by Corollary \ref{x}.

We may then assume $\h(x,y,v,b)=3$ and since $\h(x,y,v)=\h(x,y,b)=3$, we may assume $b\in (v,x,y)$ and $b\notin (x,y)$. After modifying $e$, $f$ and $v$, we may assume $b=v$. Since $F\in (y^3, y(ax+by), x,b)$, we have $fy^2\in (x,b,y^3)$, thus $f\in (x,y,b)$. Since $(x,y)^3\subseteq J$, we may assume $f\in (b)$. After possibly modifying $c$ modulo $y$ and $e$, we may further assume $f=0$. 
We now have
\begin{align*}
I^{un} &= (x,b) \cap \left[(x,y)^3 + (x,y)(ax+by) + (c(ax+by) + dx^2 + exy)\right]\\
 &= (x^3, x^2y, xy^2, x(ax+by), y(ax+by), c(ax+by) + dx^2 + exy).
 \end{align*}
If $\h(x,y,a,b,c,d,e) \ge 6$, one checks that if $L = (x^3, by^3):I^{un}$, then
\[L  = (x^{3},b y^{3},x^{2} y^{2},a x^{2} y-b x y^{2},a^{2} c x^{2}-a
       b c  x y- b d x^{2} y+ b^{2} c y^{2}+ be x y^{2}),\]
       and $\pd(R/L) = 3$. This implies $\pd(R/I) \le 4$ by Theorem~\ref{pd+1}(i).  Otherwise, all cubics in $I$ can be expressed in terms of at most $5$ variables and $\pd(R/I)\leq 5$ by Proposition~\ref{<=t}. \\

\noindent\underline{Case (vii): $J = \left(x,y\right)^3  +(x,y)(ax + by) + \left(a(ax+by)+cxy+dy^2\right)$}\\
\noindent\underline{$+ \left( b(ax+by)-cx^2-dxy \right)$, with $\h(x,y,a,b) = 4$ and $\h(x,y,a,b,c,d) \ge 5$.}\\
If $\h(x,y,a,b,c,d) = 6$, then the last two generators of $J$ cannot lie in $(u,v)$, 
so $I^{un}\subseteq \left(x,y\right)^3  +(x,y)(ax + by)$. Then all cubics in $I$ can be expressed in terms of $x,y,a,b$.  Otherwise, $\h(x,y,a,b,c,d) \le 5$, and similarly all cubics in $I$ can be written in terms of at most 5 variables. In either case $\pd(R/I) \le 5$ by Remark \ref{observ}(c). \\

\noindent\underline{Case (viii): $J  = \left(x,y\right)^3 +  (x,y)(ax+by) + [J]_{\ge 4}$,  with $\h(x,y,a,b) = 4$.}\\
In this case, all the cubics in $I^{un}$ can be expressed in terms of $x,y,a,b$, so $\pd(R/I) \le 4$ by Remark \ref{observ}(c).

\end{proof}

\subsection{Type $\boldsymbol{\langle 1,3;1,1\rangle}$}

\begin{prop}\label{L13;11}  If $L$ is of type $\langle 1,3;1,1 \rangle$, then $\pd(R/I) \le 5$.
\end{prop}

\begin{proof} By assumption, $L$ is the intersection of a height $2$ multiplicity $3$ prime ideal $P$ and a linear prime $(x,y)$. Since $k=\overline{k}$, the ideal $P$ is either generated by the $2\times 2$ minors of a $2\times 3$ matrix ${\mathbf{M}}$ of linear forms, say $P = I_2(\mathbf{M})$, or a complete intersection generated by a linear form $z$ and a cubic form $c$. (See e.g. \cite[Theorem~1]{EH2}.)  If $P = (z,c)$, then $L$ contains a quadric and we are done by Theorem~\ref{pd+1}(iii). So we may assume that $P = I_2(\mathbf{M})$; we observe that since $I_2(\mathbf{M}) \not\subseteq (x,y)$, we have $\h(P + (x,y)) = 3$ or $4$.  Since both $P$ and $(x,y)$ are Cohen-Macaulay of height $2$, if we set $Q=P+(x,y)$, then by Lemma \ref{ses} and Theorem \ref{pd+1}(i) it suffices to show $\pd(R/Q)\leq 5$.

 If $\h(Q) = 4$, then $x,y,$ is an $R/P$-regular sequence and then $\pd(R/Q) = 4$. If $\h(Q) = 3$, then the image of $P$ has height one in $R/(x,y)$, i.e. $P\subseteq (e,x,y)$ where $e$ is either a quadric or a linear form.  If $\deg(e)=2$, since $P$ is generated in degree 2, then  $Q=P + (x,y) = (x,y,e)$, so $\pd(R/Q) = 3$.  If $\deg(e) = 1$ then 
 we have $Q=P + (x,y) = (x,y, ea, eb, ec)$ for linear forms $a,b,c$.  Applying Lemma~\ref{ses} to the short exact sequence
 \[ 0 \to R/(Q:e) \to R/Q \to R/Q+(e) \to 0,\]
  one sees that $\pd(R/Q) \le 5$. This concludes the proof.
\end{proof}

\begin{prop}\label{13;11} Suppose $I^{un}$ is of type $\langle 1,3;1,1 \rangle$.  Then $\pd(R/I) \le 5$
\end{prop}

\begin{proof}
By assumption, $I^{un}$ is the intersection of a height $2$ multiplicity $3$ prime ideal $P$ and a linear prime $(x,y)$.  As above, $P$ is either generated by the $2\times 2$ minors of a $2\times 3$ matrix ${\mathbf{M}}$ of linear forms or a complete intersection generated by a linear form $z$ and a cubic form $c$.  If $P = (z,c)$, then $J$ contains a quadric and we are done by Theorem~\ref{pd+1}(iii).  So we can assume that $P = I_2(\mathbf{M})$, and as above $\h(P+(x,y))=3$ or $4$.

  If $\h(P + (x,y)) = 4$, then $x,y$ form a regular sequence on $P = I_2(M)$. 
  It follows that $I^{un} = P(x,y)$ and that $R/(P + (x,y))$ has resolution given by the tensor product of the resolutions of $R/P$ and $R/(x,y)$.  Hence $R/(P + (x,y))$ has free resolution $H_\*$ of the form:
  \[  0 \rightarrow R(-5)^2  \overset{d_4}{\longrightarrow} R^7(-4) \rightarrow R^8(-3)  \oplus R(-2) \rightarrow R^3(-2) \oplus R^2(-1) \rightarrow R.\]
In particular, $\pd(R/(P + (x,y))) = 4$.
Now consider the short exact sequence
\[ 0 \to R/I^{un} \to R/P \oplus R/(x,y) \to R/(P + (x,y)) \to 0.\]
The middle term has free resolution $G_\*$ of the form
\[  0 \rightarrow R(-3)^2\oplus R(-2) \rightarrow R^3(-2) \oplus R^2(-1) \rightarrow R^2.\]
Therefore $\pd(R/I^{un}) = 3$.  Let $F_\*$ be the minimal free resolution of $I^{un}$, so that the mapping cone $\mathrm{cone}(F_\* \to G_\*)$ is a (possibly nonminimal) free resolution of $R/(P + (x,y))$.  Comparing ranks and degrees we deduce that $F_\*$ has the form:
\[  0 \rightarrow R(-2)^5  \overset{\partial_3}{\longrightarrow} R^7(-4) \rightarrow R^6(-3) \rightarrow R.\]
It follows that $\mathrm{cone}(F_\* \to G_\*)$ is minimal beyond the $0$th homological degree and hence  $\partial_3 \cong d_4$.  
Since $P + (x,y)$ is Cohen-Macaulay of height $4$, the dual $H_\*^\ast$ of $H_\*$ is a minimal free resolution of $\coker(d_4^\ast)$.  Hence $\pd(\coker(\partial_3^\ast)) = \pd(\coker(d_4^\ast)) = 4$.  By \cite[Proposition 3.7]{HMMS3}, we have $\pd(R/I) \le 4$.

If $\h(P + (x,y)) = 3$, then by \cite[Lemma 4.1]{HMMS3}, if we consider $\mathbf{M}$ modulo $(x,y)$ it has one of the following three forms:
\begin{enumerate}
\item $\mathbf{M} = \begin{pmatrix} a&0&0\\d&e&f\end{pmatrix}$, where $a \neq 0$ and $\h(e,f) = 2$;
\item $\mathbf{M} = \begin{pmatrix} a&b&0\\d&e&0\end{pmatrix}$, where $ae-bd \neq 0$;
\item $\mathbf{M} = \begin{pmatrix} a&b&0\\d&0&b\end{pmatrix}$, where $\h(a,b,d) = 3$.
\end{enumerate}
The ideal $P$ has 3 quadric generators. However, in cases (1) and (2) the image of $P$ in $R/(x,y)$ has at most two generators; therefore $I^{un} = P \cap (x,y)$ contains at least one quadric.  Thus, $\pd(R/I) \le 4$ by Theorem~\ref{pd+1}(iii). 
In case (3), all generators of $P \cap (x,y)$ can be expressed in terms of at most $5$ variables $a,b,d,x,y$.  By Proposition~\ref{<=t}, $\pd(R/I) \le 5$.  


\end{proof}

\end{document}